\documentclass[11pt,letterpaper]{amsart}
\usepackage{amsmath,amsthm,amsfonts,amscd,amssymb,eucal,latexsym,mathrsfs,stmaryrd,enumitem,bm,mathtools}
\usepackage[dvipsnames]{xcolor}
\usepackage[hidelinks]{hyperref}
\usepackage{microtype}
\usepackage{dirtytalk} 
\usepackage{soul} 
\sethlcolor{pink}

\newtheorem{theorem}{Theorem}[section]

\newtheorem{lemma}[theorem]{Lemma}

\theoremstyle{definition}
\newtheorem{definition}[theorem]{Definition}
\newtheorem{remark}[theorem]{Remark}
\newtheorem{notation}[theorem]{Notation}

\newcounter{theoremintro}
\newtheorem{theoremi}[theoremintro]{Theorem}

\newcommand{\overbar}[1]{\mkern 1.5mu\overline{\mkern-1.5mu#1\mkern-1.5mu}\mkern 1.5mu}

\newcommand{\cZ}{{\mathcal Z}}

\newcommand{\sI}{{\mathscr I}}
\newcommand{\sK}{{\mathscr K}}

\newcommand{\sP}{{\mathscr P}}
\newcommand{\sQ}{{\mathscr Q}}

\newcommand{\sT}{{\mathscr T}}
\newcommand{\sU}{{\mathscr U}}
\newcommand{\sV}{{\mathscr V}}
\newcommand{\sW}{{\mathscr W}}

\newcommand{\fA}{{\sf A}}
\newcommand{\fF}{{\sf F}}
\newcommand{\fS}{{\sf S}}
\newcommand{\Eb}{\mathbb{E}}
\newcommand{\Cb}{{\mathbb C}}
\newcommand{\Zb}{{\mathbb Z}}
\newcommand{\Nb}{{\mathbb N}}
\newcommand{\Pb}{{\mathbb P}}

\newcommand{\eps}{\varepsilon}
\newcommand{\unit}{1}
\newcommand{\comp}{c}
\newcommand{\entry}{\bullet}

\numberwithin{equation}{section}

\DeclareMathOperator{\supp}{supp}

\DeclareMathOperator{\id}{id}

\DeclareMathOperator{\dF}{F}

\DeclareMathOperator{\trs}{tr}

\DeclarePairedDelimiter{\norm}{\lVert}{\rVert}

\allowdisplaybreaks

\begin{document}

\title{Topological full groups and stable rank one}

\author{David Kerr}
\address{David Kerr,
Mathematisches Institut,
University of M{\"u}nster, 
Einsteinstr.\ 62, 
48149 M{\"u}nster, Germany}
\email{kerrd@uni-muenster.de}

\author{Spyridon Petrakos}
\address{Spyridon Petrakos,
Department of Mathematical Sciences,
Chalmers University of Technology and University of Gothenburg,
Chalmers Tv\"{a}rgata 3, SE-412 96 G\"{o}teborg, Sweden}
\email{petrakos@chalmers.se}

\date{\today}

\begin{abstract}
We establish stable rank one for the reduced group C$^*$-algebras of the C$^*$-simple topological full groups 
and dynamical alternating groups constructed by Kerr and Tucker-Drob. 
The proof relies on both type II$_1$ and type III
phenomena, in the first case via the use of F{\o}lner towers and in the second via Ozawa's recent results on selflessness 
as applied to direct products of  
free products.
\end{abstract}

\maketitle

\section{Introduction}

Despite the deep and productive analogies that take hold between von Neumann algebras and C$^*$-algebras
when these become highly noncommutative, the study of simple C$^*$-algebras remains beset by many complications 
of a topological nature. The comparability and divisibility of projections that are both automatic in factors of type II or III cannot even be guaranteed in approximate form in simple unital C$^*$-algebras, and so one must postulate strict comparison and almost divisibility as properties whose validity must be tested for in examples.

Even a ``type II$_1$'' expression of zero-dimensionality like stable rank one (i.e., the density of invertible elements)
can fail in simple stably finite AH algebras \cite{Vil99}. 
All of these attributes would seem (with the necessary restriction of stable finiteness for stable rank one)
to be prerequisites for the type of fine-scale analysis 
that a Connes-type structure and classification theory would demand, 
and indeed if one limits one's scope to the ``II$_1$ factor'' setting of simple separable finite C$^*$-algebras, 
as we do in this paper, then one can 
derive all of them from the McDuff-like property of $\cZ$-stability 
that has become a cornerstone of the Elliott classification program \cite{Ror04}\footnote{For strict comparison and almost divisibility the assumption of finiteness is not necessary.}.

Absent $\cZ$-stability, it becomes a challenge to assay a simple unital finite C$^*$-algebra
for any of the properties of strict comparison, almost divisibility, and stable rank one.
In practice stable rank one has been more amenable to verification
but also plays a somewhat less central (if sometimes surprising) role in issues around classification,
while strict comparison has proven especially recalcitrant outside the scope of nuclearity, although like stable rank one it 
can be directly monitored and controlled in simple AH algebras and fails to hold for some examples within this class \cite{Vil98}
(see \cite{SchTikWhi25} for a survey). 
These differences are reflected in the partial two-out-of-three sociology that has been observed 
among the three properties: for simple separable unital nonelementary finite
$C^*$-algebras, almost divisibility and stable rank one each follow from the other two properties 
\cite[Corollary~8.12]{Thi20}\cite[Corollary~1.3]{Lin25}.

The present paper is concerned with stable rank and strict comparison in simple unital reduced
group C$^*$-algebras. These C$^*$-algebras possess a unique tracial state \cite{BreKalKenOza17} and are never nuclear
(the simplicity is in fact equivalent to the nonexistence of amenable uniformly recurrent subgroups \cite{Ken15})
and are thus natural and interesting test cases to consider as one crosses the threshold into nonamenability.
They are also very different in nature, at least from the viewpoint of their construction,
from the tracial C$^*$-algebras that arise as reduced crossed products of topologically free minimal actions
on compact metrizable spaces. Such crossed products are nuclear as soon as the group is amenable,
and in that setting one can frequently verify one or both of the following two Følner tiling properties,
the first of which implies stable rank one \cite{LiNiu20} and the second of which implies $\cZ$-stability \cite{Ker20}:
\begin{enumerate}[label=(\roman*)]
\item\label{I-URPC} URPC (uniform Rokhlin property $+$ comparison), which combines
dynamical comparison with the existence of open castles with Følner shapes and remainders
that are are uniformly small in measure,

\item\label{I-almost finite} almost finiteness, which is the same as URPC but with the additional condition that the tower levels be small in diameter.
\end{enumerate}
For a large class of amenable acting groups, including elementary amenable groups and groups of subexponential growth, 
almost finiteness is known to hold as soon as the action is free and the phase space
has finite covering dimension \cite{KerNar21,Nar24,DowZha23,NarPet25}, but the picture is most complete for
free minimal $\Zb^d$-actions, where URPC is automatic (this was shown in \cite{Niu24} using a tiling result from \cite{GutLinTsu16})
and almost finiteness, the small boundary property, and mean dimension zero are all equivalent 
(for the equivalence of the first two combine \cite{Nar22} and Theorem~6.1 of \cite{KerSza20} and for the equivalence
of the last two see \cite{GutLinTsu16}).
Moreover, when the mean dimension is nonzero it is possible for the crossed product not to have strict comparison \cite{GioKer10}.
This gives the impression that stable rank one is a more robust condition than strict comparison, 
at least for simple crossed products coming from actions on spaces, where the possibility remains open that
the stable rank is always equal to one.
Going further in this direction, recent work of Bell, Geffen, and the first author shows that Følner-like tower structure can 
also be applied through a Baire category lens to establish stable rank one in some nonnuclear tracial crossed products, 
notably those coming from a generic weakly mixing topological free minimal action of a free group
on the Cantor set preserving at least one Borel probability measure. What happens generically in the free group case  
is that one can build Følner towers from one of the generators in a way that 
confines the freeness of the relations with the other generators to much smaller dynamical scales,
thereby ensuring the approximate invariance of the towers with respect to all generators.
This Baire category maneuver allows one to bypass the unresolved problem of whether
any of these actions have dynamical comparison, which is a critical ingredient in \ref{I-URPC} and \ref{I-almost finite} above.

It would seem that these ``type II$_1$'' Følner techniques should not be applicable in the setting of simple
group C$^*$-algebras if we understand C$^*$-simplicity to be a kind of saturation by nonamenability.
Indeed progress to date has relied in one way or another on group-geometric phenomena of a type III flavour, 
specifically ones for which the prototype is precisely the freeness that we were wanting to suppress 
in the dynamical framework above. The pioneering step in this direction was
Dykema, Haagerup, and Rørdam's use of a rapid-decay-type argument to establish stable rank one
for the reduced group C$^*$-algebras of nonelementary free products, and in particular of free groups \cite{DykHaaRor97}.
The class of stable rank one groups was subsequently expanded so as to eventually come to include all acylindrically 
hyperbolic groups \cite{DykHar99,GerOsi20,Rau25}. Until very recently not much was known 
concerning strict comparison, including the question of whether $C^*_\lambda (F_2 )$ has it.
Rørdam observed early on that free probability arguments from \cite{DykRor00} can be applied to deduce
strict comparison for nontrivial infinite free products such as $F_\infty$ 
(see Proposition~6.3.2 of \cite{Rob25}). 
One can also easily come up with 
recursive constructions that will produce
a simple reduced group C$^*$-algebra which is $\cZ$-stable, although verifying $\cZ$-stability
in less artificial examples becomes extremely difficult due to the absence of Matui--Sato technology in the nonnunclear world
(compare the use of \cite{HirOro13} in the derivation of $\cZ$-stability from almost finiteness in Theorem~12.4 in \cite{Ker20}) 
and it will already even fail if the group is not inner amenable.
A breakthrough came in 2024 when 
Amrutam, Gao, Kunnawalkam Elayavalli, and Patchell proved that the reduced C$^*$-algebra of free groups,
and more generally of all acylindrically hypebolic groups with trivial finite radical and rapid decay,
satisfy the free absorption property of selflessness that was introduced by Robert and observed by him to imply
strict comparison via the above-mentioned result of Rørdam on infinite free products \cite{Rob25}.
Ozawa later identified dynamical and combinatorial conditions
that imply selflessness and that can be applied in a wide range of situations, 
including all nonelementary free products.
We also now know, owing to Lin's theorem from \cite{Lin25} cited earlier and the fact that
C$^*$-simplicity entails unique trace, that, for simple unital reduced group C$^*$-algebras,
strict comparison implies stable rank one.

Our aim in this paper is to address the family of C$^*$-simple groups constructed by the first author and Tucker-Drob
in Theorem~8.7 of \cite{KerTuc23}
as topological full groups and dynamical alternating groups of topologically free minimal subshift actions
of amenable groups on the Cantor set. Many of the dynamical alternating groups in this collection
are simple and finitely generated, properties that follow as soon as the group is finitely generated and
the action is topologically free, minimal, and expansive (i.e., a subshift over a finite alphabet) \cite{Nek17}. The assumption on the
acting group that supports the construction in \cite{KerTuc23} (a tiling property called {\it property ID}) 
incorporates an entropy condition that is not directly relevant for us here but was included there so as 
to be able to produce an uncountable infinity of examples for a given acting group by varying the dynamical entropy, which 
is an invariant of continuous orbit equivalence and consequently also of the dynamical alternating group.
We will recall the construction in its more general form in Section~\ref{S-family}.
It was inspired by the original $\Zb^2$ example in \cite{EleMon13}, 
and, like that example, arranges for nonamenability by using an idea of van Douwen that embeds the free product 
$\Zb_2 * \Zb_2 * \Zb_2$ using a ping-pong argument over a full shift along a copy of $\Zb$ in the group \cite{vDo90}.
See also \cite{Szo21} for a related construction that geometrically replicates a full shift in a way that 
applies to all non-virtually cyclic amenable groups.
When the acting group is virtually cyclic, the presence (real or virtual) of a ``one-dimensional'' full shift 
that unites all of these examples is 
incompatible with minimality, and it is a remarkable fact that free minimal actions of virtually cyclic groups on the Cantor set
always produce an amenable topological full group \cite{JusMon13,Szo21}.

In \cite{KerTuc23} the first author and Tucker-Drob employed Følner towers and a probabilistic argument inspired by \cite{KecTsa08}
to show that, for every minimal topological free action of a countable amenable group on the Cantor set,
the group von Neumann algebras of the topological full group and dynamical alternating group
have property Gamma. Later the two authors of the present paper discovered a non-probabilistic approach 
(but still using Følner towers) that strengthens the property Gamma conclusion to McDuffness.
We do not know however whether the reduced group C$^*$-algebra of any of these groups is $\cZ$-stable. 
We also do not know whether these group C$^*$-algebras are ever selfless or have strict comparison.
On the other hand we have been able to establish stable rank one for the 
C$^*$-simple examples from \cite{KerTuc23}. The underlying actions all possess
a property we call having a {\it specification ridge}
(Definition~\ref{D-specification}), which is sufficient to reach the conclusion:

\begin{theoremi}\label{T-stable rank one}
Let $\Gamma$ be a torsion-free countably infinite amenable discrete group and let $X$ be the Cantor set.
Let $q$ be an integer greater than $3$ and
let $\Gamma\curvearrowright X \subseteq \{ 1,\dots , q\}^\Gamma$ 
be a minimal topologically free right subshift action with a specification ridge.
Let $G$ be a subgroup of the topological full group $\fF (\Gamma,X)$
containing the alternating group $\fA (\Gamma , X)$. Then the reduced group C$^*$-algebra $C^*_\lambda (G)$
has stable rank one.
\end{theoremi}

The reduced group C$^*$-algebras in the theorem are all simple, as explained in Remark~\ref{R-simple}.

Our strategy for establishing Theorem~\ref{T-stable rank one}
is to simulate the dynamical property of square divisibility from \cite{BelGefKer25} in a way that will 
similarly yield stable rank one
via a Rørdam-type argument \cite{Ror91} that involves 
the near block diagonalization of zero divisors of a special form (in the spirit of \cite{LiNiu20}) 
and unitary rotation to a nilpotent element.
We proceed as in \cite{KerTuc23} by using a ``striated'' Følner tower for the $\Gamma$-action to create
a permutational Bernoulli structure whose phase space lives within $C^*_\lambda (G)$ as a spectral object
(one may think of it as an ``exponentiation'' of the Følner tower along the striation) 
and then applying some probability theory
to produce a uniform near partition of the Bernoulli space that is approximately invariant
under a prescribed finite set of elements in $G$.
We require the near partition 
to be much finer than the two-member partition that sufficed for the purposes of \cite{KerTuc23},
and so we need to manufacture the Bernoulli structure within a permutational wreath product
whose base is a very large finite alternating group, like the ones used in \cite{KerPet25}.
To produce the ersatz square divisibility
some care is required in the construction, manipulation, and double-indexing of the near partition,
in particular to negotiate issues around approximate invariance and zero division that are either 
not present or handled differently in \cite{BelGefKer25}.

The point where our situation critically departs from \cite{BelGefKer25} is that we cannot replicate the subequivalences 
in the definition of square divisibility in a dynamical way within the Bernoulli structure, or even C$^*$-algebraically within
the ambient wreath product.
This is due to the presence of the trivial representation of the ambient wreath product,
which sits underneath one of the sets in the Bernoulli space and cannot be budged
within the group algebra of the wreath product, where it appears as a one-dimensional direct summand.
If the entire group $G$ were amenable, its trivial representation would also
sit underneath the same set, 
which we could then not even move within 
$C^*_\lambda (G)$ itself, causing the whole argument to break down.
This gives a nice dynamical illustration of how, in the context of groups, 
nonamenability is related to the ability to move around projections and other positive elements in the C$^*$-algebra.
In its most permissive form this transportability of positive elements 
is associated with C$^*$-algebraic simplicity (although strict comparison still sometimes fails in this situation), 
and indeed it remains unclear how much we can move around the trivial 
representation of the wreath product 
if we merely assume $G$ to be nonamenable.\footnote{It was shown in \cite{Sca23} that, for minimal actions, nonamenability of the 
dynamical alternating group is actually equivalent to its C$^*$-simplicity as well as to the C$^*$-simplicity of the topological full group,
but we have been unable to exploit this fact or its proof.}
It is at this juncture that we invoke C$^*$-simplicity through the specific form of the construction
that was devised to produce it in \cite{KerTuc23}. 
Using a variation of the van Douwen ping-pong argument,
we embed into $G$ a collection of nonelementary free products whose first factors are the copies of
the alternating group forming the base of the wreath product. These embedded free products are configured so as to
have pairwise disjoint supports in the topological full group, so that their direct product $G_0$ also
forms a subgroup of $G$. The group C$^*$-algebra $C^*_\lambda (G_0)$ will then
be selfless by the previously mentioned work of Ozawa \cite{Oza25}. We can consequently exploit
the strict comparison that ensues from selflessness to realize the desired subequivalences 
C$^*$-algebraically within $C^*_\lambda (G_0)$, which is sufficient
to generate the near block diagonalization that leads to stable rank one. 

The proof of Theorem~\ref{T-stable rank one} appears in Section~\ref{S-theorem} and relies
on the lemmas established in Sections~\ref{S-free products}, \ref{S-zero division}, and \ref{S-combinatorial}.
Section~\ref{S-preliminaries} reviews some definitions and basic facts concerning topological full groups,
while Section~\ref{S-family} describes the family of actions that are the subject of Theorem~\ref{T-stable rank one}.
Section~\ref{S-construction} describes the construction of striated towers that will be needed 
both in Section~\ref{S-zero division} and in the proof of Theorem~\ref{T-stable rank one} in Section~\ref{S-theorem}.
\medskip

\noindent{\it Acknowledgements.}
This project was supported by the Deutsche Forschungsgemeinschaft 
(DFG, German Research Foundation) under Germany's Excellence Strategy EXC 2044/2-390685587, 
Mathematics M{\"u}nster: Dynamics--Geometry--Structure, by the SFB 1442 of the DFG,
by a Simons Foundation grant (award no.\ SFI-MPS-T-Institutes-00010825), and 
by State Treasury funds as part of a task commissioned by the Minister of Science and Higher Education under the project “Organization of the Simons Semesters at the Banach Center - New Energies in 2026-2028” (agreement no.\ MNiSW/2025/DAP/491).
The second author was funded by the Knut and Alice Wallenberg Foundation through a postdoc grant.

\section{Some general notation and terminology}

The identity element of a group will always be denoted by $e$. We write $A\Subset B$ to mean that $A$ is a finite subset of $B$. The symmetric and alternating groups over a finite set $F$ are written $\fS_F$ and $\fA_F$, respectively.

Let $E\subseteq F$ and $Y$ be finite sets. We denote by $\pi_E$ the coordinate projection map
$Y^F \to Y^E$. We write $A\subseteq_E Y^F$ if $A$ is a subset of $Y^F$
with the property that $A = \pi_E^{-1} (\pi_E (A))$, i.e., membership of a point in $A$ is determined by its
coordinates over $E$. We refer to a point in $Y^F$ as a {\it configuration} over $F$.
A {\it cylinder set} $A$ in $Y^F$ is a subset of the form $\{ y\in Y^F : y|_E = w \}$
for some $E\Subset F$ and $w\in Y^E$, in which case $E$ is called the {\it window} 
and $A$ is said to be determined by $w$. We will frequently use without notational comment
the canonical action of the of the symmetric group $\fS_F$ on $Z^F$,
as defined by $\sigma (z_s )_{s\in F} = (z_{\sigma^{-1} (s)})_{s\in F}$. 

For a group $\Gamma$ and a $q\in\Nb$ we define the right shift action 
$\Gamma\curvearrowright\{ 1,\dots  q\}^\Gamma$ by $(sx)(t) = x(ts)$ for all $s,t\in\Gamma$ 
and $x\in\{ 1,\dots , q\}^\Gamma$. The restriction of such an action to a closed $\Gamma$-invariant
subset we call a {\it right subshift action}.

For a discrete group $G$, the group ring $\Cb G$ is viewed as a $^*$-subalgebra
of the reduced group C$^*$-algebra $C^*_\lambda (G)$ and accordingly its elements have the form
$\sum_{g\in L} \beta_g u_g$ for some $L\Subset G$ and scalar coefficients $\beta_g$, where 
$u_g$ for $g\in G$ are the canonical unitaries.
The {\it support} of an element $b$ in $\Cb G$ is the minimal set $L\Subset G$ 
for which $b$ can be expressed as $\sum_{g\in L} \beta_g u_g$. 
The left regular representation of $G$ induces a faithful tracial state on $C^*_\lambda (G)$,
which we will invariably denote by $\tau$.

\section{Preliminaries on towers and topological full groups}\label{S-preliminaries}

Let $\Gamma\curvearrowright X$ be an action of a countable discrete group on the Cantor set
(all such actions are assumed to be continuous). The {\it topological full group} of the action,
written $\fF (\Gamma , X)$, is the discrete group of all homeomorphisms $h:X\to X$ 
for which there is a clopen partition $X = A_1 \sqcup\cdots\sqcup A_n$
and $s_1 , \dots , s_n \in\Gamma$ such that $hx = s_i x$ for all $i=1,\dots , n$ and $x\in A_i$.
This group is countable given that $\Gamma$ is countable and there are only countably many clopen partitions of $X$. 
The {\it support} of a element $h\in\fF (\Gamma , X)$ is the clopen set of all $x\in X$ such that $hx\neq x$.

By a {\it tower} for the action $\Gamma\curvearrowright X$ we mean a pair $(S,B)$ where $S$ is a finite subset of $\Gamma$ (the {\it shape})
and $B$ is a subset of $X$ (the {\it base}) such that the sets $sB$ for $s\in S$ (the {\it levels}) are pairwise disjoint.
The tower is {\it clopen} if $B$ is clopen, which by the continuity of the action is equivalent to all of the levels being clopen.

Canonically associated to a clopen tower $(S,B)$ is the embedding $\sigma\mapsto h_\sigma$ of the symmetric group 
over $S$ into $\fF (\Gamma , X)$ defined by
$h_\sigma x = \sigma (s)s^{-1} x$ for every $s\in S$ and $x\in sB$ and $h_\sigma x = x$ for every 
$x\in X\setminus SB$. 
We write $\fS (S,B)$ for the subgroup of $\fF (\Gamma , X)$ obtained under this embedding.
We also thereby obtain a copy $\fA (S,B) \subseteq \fS (S,B)$ of the alternating group over $S$.
The subgroup of $\fF (\Gamma , X)$ generated by 
such embedded finite alternating groups over all clopen towers with three levels is called the {\it dynamical alternating group}
and denoted by $\fA (\Gamma , X)$. This was introduced by Nekrashevych in \cite{Nek17}, where the 
following facts were established.
If the action has no finite orbits (in particular, if $\Gamma\curvearrowright X$ is minimal) then
$\fA (\Gamma , X)$ is equal to the subgroup
generated by the embedded finite alternating groups arising from all clopen towers.
The dynamical alternating group is contained in the commutator subgroup of $\fF (\Gamma , X)$
and is equal to it when the action is almost finite (this relies on \cite{Mat12}, as explained in Section~4 of \cite{Nek17}).
If the action $\Gamma\curvearrowright X$ is minimal then $\fA (\Gamma , X)$ is simple,
while if $\Gamma$ is finitely generated and $\Gamma\curvearrowright X$ has no orbits of
cardinality less than $5$ then $\fA (\Gamma , X)$ is finitely generated.

\section{The family of actions}\label{S-family}

For the proof of stable rank we will need to assume that our subshift actions possess a 
localized specification property contingent on
C$^*$-simplicity. The prototypes are the actions constructed in the proof 
of Theorem~9.7 in \cite{KerTuc23}. We will first recall the basic parameters of this construction
and then abstract the specification property that will be essential for the proof of Theorem~\ref{T-stable rank one}.

A {\it tiling} of $\Gamma$ is a finite collection $\sT = \{ (S_i ,C_i ) \}_{i\in I}$ of pairs consisting
of a finite set $S_i \subseteq \Gamma$ (a {\it shape}) and a set $C_i \subseteq\Gamma$ (a {\it centre set})
such that $\Gamma$ partitions as $\bigsqcup_{i\in I} \bigsqcup_{c\in C_i} S_i c$. 
We refer to $S_i c$ as a {\it tile} and to $C_i$ as the set of {\it tiling centres} for the shape $S_i$.

A sequence $(\sT_n )$ of tilings of $\Gamma$ is {\it tightly nested} if for every $n>1$ and
every pair $(S,C)$ in $\sT_n$ there is a partition $S = \bigsqcup_{j\in J} S_j d_j$ where 
the sets $S_j$ are shapes of $\sT_{n-1}$ and the $d_j$ are elements of $\Gamma$ such that 
for every $c\in C$ and $j\in J$ the element $cd_j$ is a tiling centre for $S_j$.
It is a {\it Følner tiling sequence} if for every finite set $K\subseteq\Gamma$ and $\delta > 0$
there is an $n_0 \in\Nb$ such that, for every $n\geq n_0$, each shape $S$ of $\sT_n$ is 
$(K,\delta )$-invariant in the sense that $|tS\Delta S| \leq \delta |S|$ for every $t\in K$.

Suppose now that $\Gamma$ is nontorsion and let $H$ be an infinite cyclic subgroup of $\Gamma$.
Let $(\sT_n )$ be a tightly nested Følner tiling sequence satisfying the following:
\begin{enumerate}[label=(\roman*)]
\item\label{I-centre} the centre sets of each tiling $\sT_n$ are syndetic,

\item\label{I-density} for every $\delta > 0$ there is an $n_0 \in\Nb$ such that $|H\cap T| \leq \delta |T|$
for every tile $T$ belonging to a tiling $\sT_n$ with $n \geq n_0$,

\item\label{I-cyclic} for every $n\in\Nb$ there are a pair $(S,C)\in \sT_n$ and a subgroup $H_0$ of $H$ 
contained in $C$ such that $H\subseteq \bigsqcup_{c\in H_0} Sc$.
\end{enumerate}
Observe that every subsequence of $(\sT_n )$ is a tightly nested Følner tiling sequence satisfying the same conditions.
Condition~\ref{I-centre} will enable us to arrange minimality in our subshift construction,
while conditions~\ref{I-density} and \ref{I-cyclic} will permit us to do this in such a way
that we can embed a full shift along $H$.
Note that \ref{I-density} cannot hold if $\Gamma$ is virtually cyclic, 
and if we assume $\Gamma$ to be finitely generated and not
virtually cyclic, then \ref{I-density} can actually be derived from \ref{I-cyclic} and the Følner hypothesis
on the tiling, as can be seen using the fact that $H$ has infinite index in $\Gamma$ in this case.

Let $q$ be an integer greater than $1$. 
We have the right shift action $\Gamma\curvearrowright \{ 1,\dots , q \}^\Gamma$
defined by $(sx)(t) = x(ts)$ for all $s,t\in\Gamma$ and $x\in\{ 1,\dots , q\}^\Gamma$.
We construct a family of minimal subshift actions 
$\Gamma\curvearrowright X\subseteq \{ 1,\dots , q \}^\Gamma$ as follows.

For $L\subseteq\Gamma$ we write $\pi_L : \{ 1,\dots , q \}^\Gamma \to \{ 1,\dots , q \}^L$ for the
coordinate projection map. A {\it box} is a subset of $\{ 1,\dots , q \}^\Gamma$ of the form
$\prod_{t\in\Gamma} A_t$ where each $A_s$ is a subset of $\{ 1,\dots , q \}$. For every
box $A = \prod_{t\in\Gamma} A_t$ and $T\subseteq\Gamma$ we set
\[
D_T (A) = \{ t\in T : A_t = \{ 1,\dots ,q \} \} .
\]
Let $A_0 = \{ 1,\dots , q\}^\Gamma \supseteq A_1 = \prod_{t\in\Gamma} A_{1,t}
\supseteq A_2 = \prod_{t\in\Gamma} A_{2,t} \supseteq\dots$ be a decreasing 
sequence of boxes such that the following hold for every $n\in\Nb$:
\begin{enumerate}[label=(\roman*),resume]
\item\label{I-pattern} for every $y\in A_{n+1}$, $x\in A_n$, $(S,C)\in\sT_n$, tile $T$ of $\sT_{n-1}$, and $c\in C$ there exists a $\tilde{c} \in\Gamma$
such that $T\tilde{c} \subseteq Sc$ and $y_{t\tilde{c}} = x_t$ for all $t\in T$,

\item\label{I-cd} for each $(S,C)\in\sT_n$ one has $A_{n,sc} = A_{n,sd}$ for all $s\in S$ and $c,d\in C$,

\item\label{I-H} $H\subseteq D_\Gamma (A_n )$,

\item\label{I-F} $\pi_{F_{n-1}} (A_n ) = \pi_{F_{n-1}} (A_{n-1} )$ where $F_{n-1}$ is the 
tile of $\sT_{n-1}$ containing $e$.
\end{enumerate}
Such sequences are constructed in the proof of Theorem~9.7 in \cite{KerTuc23} by a recursive
procedure that involves passing to a tiling subsequence and defining the boxes $A_{n+1}$ to have ever increasing
density of singleton factors that collectively replicate the patterns of points in $A_n$ on large finite windows
(in the limit over $n$ this will create a sufficient amount of recurrence to achieve minimality). Although this density
decreases, we can prevent it from going to zero, so that the actions we next define will be nontrivial.

Put $A = \bigcap_{k=1}^\infty A_k$, which is nonempty and compact since each
$A_k$ is. We then have the right subshift action 
$\Gamma\curvearrowright X:=\overline{\Gamma A} \subseteq \{ 1,\dots , q \}^\Gamma$.
As in the proof of Theorem~9.7 in \cite{KerTuc23}, one can verify using syndeticity and conditions~\ref{I-cd} and \ref{I-F}
that the action $\Gamma\curvearrowright X$ is minimal. 
It is not clear however whether it is topologically free. It is shown in
Section~9 of \cite{KerTuc23} that if the 
entropies of the tilings $\sT_n$ converge to zero (i.e., the tiling sequence $(\sT_n )$
fully witnesses property ID in Definition~9.1 of \cite{KerTuc23}),
then by carefully controlling the sets $D_T (A_k )$ one can arrange for the action to  
have any desired value of topological entropy between $0$ and $\log q$.
Without the entropy hypothesis on the $\sT_n$ one can still control the sets $D_T (A_k )$
so as to guarantee that the topological entropy is nonzero.
The point of this is that minimality and nonzero topological entropy together
imply that the action is topologically free, and even free if $\Gamma$ is torsion-free
(a fact recorded as Lemma~9.6 in \cite{KerTuc23}).
Moroever, if we know $\Gamma\curvearrowright X$ to be topologically free,
then $X$ cannot have isolated points due to minimality and the infiniteness of $\Gamma$,
and so it must be the Cantor set.

What is essential for the proof of Theorem~\ref{T-stable rank one} is that the above subshift actions
satisfy the following specification property with respect to $H$, as is a consequence of (v)-(vii) above.
Property (vii) means that the product of boxes over the tiles intersecting $H$ at a given stage $n$
survive at all subsequent tiling stages and thus give rise to the desired specifiability of configurations over
progressively larger thickenings of $H$ at larger and larger tiling scales. In \cite{KerTuc23} this is used to 
show that, when $q\geq 4$, the rigid stabilizers of open subsets each contain a copy of $\Zb_2 * \Zb_2 * \Zb_2$, 
which implies C$^*$-simplicity. In our case we will use it 
to embed copies of $\fS_d * \Zb_2$ (Lemma~\ref{L-free products}).

\begin{definition}\label{D-specification}
Let $\Gamma\curvearrowright X \subseteq \{ 1,\dots , q \}^\Gamma$ be a right subshift action
and $H=\langle a\rangle$ an infinite cyclic subgroup of $\Gamma$.
We say that the action has a {\it specification ridge along $H$} if for every $E\Subset \Gamma$ 
there exist a $T\Subset\Gamma$ with $E\cup\{ e\} \subseteq T$, a $c\in H$, and sets $A_r \subseteq \{ 1,\dots , q \}$ 
for $r\in R:= \bigcup_{n\in\Zb} Tc^n$ for which
\begin{enumerate}
\item the sets $Tc^n$ for $n\in\Zb$ are pairwise disjoint and cover $H$,

\item $A_{tc^n} = A_t$ for all $t\in T$ and $n\in\Zb$,

\item $A_{a^n} = \{ 1,\dots , q \}$ for every $n\in\Zb$,

\item $X |_R \supseteq \prod_{r\in R} A_r$.
\end{enumerate}
We also simply say that the action has a {\it specification ridge} if it has a specification ridge
along some such subgroup $H$.
\end{definition}

\begin{remark}\label{R-simple}
For subshift actions $\Gamma\curvearrowright X \subseteq \{ 1,\dots , q \}^\Gamma$ with $q\geq 4$
that have a specification ridge, every subgroup of
$\fF (\Gamma ,X)$ containing $\fA (\Gamma ,X)$ is C$^*$-simple. The proof of Theorem~8.7 in \cite{KerTuc23} 
shows in this case that the rigid stabilizer of every nonempty open set contains a copy of the free group
$F_2$ and hence is nonamenable, and this implies C$^*$-simplicity by a result of Le Boudec and Matte Bon \cite{LeBMat18}.
The nonamenability of the rigid stabilizers also follows from Lemma~\ref{L-free products} in the next section.
\end{remark}

\section{Embedding free products}\label{S-free products}

\begin{lemma}\label{L-free products}
Let $q\geq 4$ and let $\Gamma\curvearrowright X \subseteq \{ 1,\dots , q \}^\Gamma$ be a right subshift action
with a specification ridge.
Let $U\subseteq X$ be a nonempty open set and $d\in\Nb$. 
Then there exist a clopen tower $(S,B)$ with $|S|=d$ and $SB\subseteq U$, and an involution $g\in\fA(\Gamma ,X)$ 
supported in $SB$ such that $\fS(S,B)$ is free 
with respect to $g$.
\end{lemma}

\begin{proof}
The set $U$ is open and thus contains a nonempty subset that is the intersection of a cylinder set
in $\{ 1,\dots , q \}^\Gamma$ with $X$. Having a specification ridge guarantees that all such subsets are infinite.
The cylinder set is determined by a configuration $z\in\{1,\dots,q\}^K$ over some window $K\Subset\Gamma$.

Take an infinite cyclic subgroup $H=\langle a\rangle\subseteq\Gamma$ along which the action has
a specification ridge.
Accordingly find $e\in T\Subset\Gamma$, $c\in h$, and sets $A_r$ fulfilling the conditions in the definition of specification ridge with $E$ taken there to be $K\cup\{a^k:0\leq k\leq|K|+1\}$. 
A configuration $x\in X$ will be called \emph{stamped} if $x|_K=z$,
the importance of this condition being that it implies that $x$ lies in $U$. 
We also fix $0\leq m_1 < m_2\leq|K|+1$ with $a^{m_i}\notin K$ (these will be our control coordinates, 
which we will be freely allowed to choose in $\{1,\dots,q\}$).

Consider now the clopen set $B$ consisting of points $x$ in $X$ such that
\begin{enumerate}
\item $c^kx$ is stamped for $0\leq k\leq d-1$,
\item $x(a^{m_1})=2$,
\item $x(a^{m_1}c^k)=1$ for $1\leq k\leq d-1$.
\end{enumerate}
Then $(S:=\{c^k:k=0,\dots,d-1\},B)$ forms a tower with $|S| =d$ and $SB \subseteq U$,
with $B$ being nonempty by the definition of specification ridge.
To construct a nontrivial involution $g$ that is free with respect to $\fS(S,B)$, we consider the clopen set
\[
D=SB\cap c^{-3d}SB\cap c^{-6d}SB\cap c^{-9d}SB\cap L
\]
where $L$ denotes the set of points $x$ in $X$ with
\begin{align*}
x(a^{m_2})=1,\hspace*{4mm}
x(a^{m_2}c^{3d})=2,\hspace*{4mm}
x(a^{m_2}c^{6d})=3,\hspace*{4mm}
x(a^{m_2}c^{9d})=4.
\end{align*}
This set is nonempty by the definition of specification ridge.
Note moreover that the sets $D$, $c^{3d} D$, $c^{6d} D$, and $c^{9d} D$ are pairwise disjoint by the definition of $L$,
and so we can define a nontrivial involution $g$ by setting, for all $x\in D$,
\begin{align*}
gx=c^{9d}x,\hspace*{4mm}
gc^{9d}x=x,\hspace*{4mm}
gc^{6d}x=c^{3d}x,\hspace*{4mm}
gc^{3d}x=c^{6d}x,
\end{align*}
and letting $g$ be the identity everywhere else. Since $g$ acts on its support by swapping each of two disjoint pairs of clopen 
sets that are all images of each other under elements of $\Gamma$, it belongs to $\fA(\Gamma,X)$.
Note also that the support of $g$ lies in $SB$ by the definition of $D$.

To establish the desired freeness,
suppose that we are given an element $w$ in $\fF (\Gamma ,X)$
of the form $h_ngh_{n-1}g\cdots gh_1$, where each $h_i$ belongs to 
$\fS(S,B)\setminus\{e\}$, and let us show that it is nontrivial. The verification of this nontriviality for nontrivial words of the other
three forms (depending on the membership of the beginning and ending letter) can be carried out similarly.
For every $i\in\{1,\dots,n\}$ there exist $0\leq l_i<r_i\leq d-1$ such that $h_i$ acts as $c^{r_i-l_i}$ on $c^{l_i}B$. 
We set $j_i=r_i-l_i$ and $J_k=\sum_{i=1}^kj_i$. The conditions on the set $A_r$ in the definition
of specification ridge permit us to find a point $y\in X$ satisfying the following:
\begin{enumerate}[resume]
\item $c^k y$ is stamped for $-d\leq k\leq 10nd$,
\item $y(a^{m_1}c^{J_k+9kd-l_{k+1}})=2$ for $0\leq k\leq n-1$,
\item $y(a^{m_1}c^{J_k+9kd-l_{k+1}+s})=1$ for $0\leq k\leq n-1$ and $1\leq s\leq d-1$,
\item $y(a^{m_1}c^{J_k+9(k-1)d+3td})=2$ for $0\leq k\leq n-1$ and $0\leq t\leq 2$,
\item $y(a^{m_1}c^{J_k+9(k-1)d+3td+s})=1$ for $0\leq k\leq n-1$, $0\leq t\leq 2$, and $1\leq s\leq d-1$,
\item $y(a^{m_2}c^{J_k+9(k-1)d+3td})=t+1$ for $0\leq k\leq n$ and $0\leq t\leq3$.
\end{enumerate}
Then
\[
y(a^{m_2})=4\neq 1=(wy)(a^{m_2}),
\]
showing that $w$ is nontrivial.
\end{proof}

\section{Construction of striated towers}\label{S-construction}

Let $\Gamma\curvearrowright X \subseteq \{ 1,\dots , q \}^\Gamma$ 
be a minimal topologically free right subshift action with a specification ridge.

\begin{definition}\label{D-striated}
By a {\it striated clopen tower} we mean a triple $(F,S,C)$ where $F$ and $S$ are finite subsets of $\Gamma$
such that the map $(t,s) \mapsto ts$ from $F\times S$ to $\Gamma$ is injective and 
$C$ is a clopen subset of $X$ such that $(FS,C)$ forms a tower.
\end{definition}

Starting with a finite set $\Omega$ of elements in $\fF(\Gamma,X)$ we will describe a construction of a
striated clopen tower $(F,S,C)$ in $X$ along with an associated permutational Bernoulli structure.
We will refer to this as {\it running the construction for $\Omega$}.
It will provide the framework for our arguments in Section~\ref{S-zero division} and \ref{S-theorem}. 
There is some flexibility in choosing the sets $F$ and $S$ 
and the phase space of the Bernoulli structure, and indeed these will need to be further specified 
when carrying out the proof of Theorem~\ref{T-stable rank one} in Section~\ref{S-theorem}.

By definition, for every $g\in\fF(\Gamma,X)$ there is a continuous function $x\mapsto\theta (g,x)$ on $X$ such that
$gx = \theta (g,x)x$ for all $x\in X$, and this determines a (finite) clopen partition $\sQ_g$ of $X$ whose members
are those nonempty sets of the form $\{ x\in X : \theta (g,x) = s \}$ for $s\in \Gamma$.
Write $K$ for the symmetric subset of $\Gamma$ that is the union of the set 
\[ 
\bigcup_{g\in\Omega} \bigcup_{x\in X} \theta (g,x) \cup \{ e \}
\] 
with its inverse.

Write $\sQ$ for the join of the clopen partitions $\sQ_g$ for $g\in\Omega$. Then there exists a map
\[
P : [\Omega ]^2 \to \sQ
\]
on the collection of two-element subsets of $\Omega$
such that the constant values of the two functions $x\mapsto \theta (g,x)$ and $x\mapsto\theta (g',x)$ on $P_{\{g,g'\}}$
are different. 
Since the action $\Gamma\curvearrowright X$ is minimal and topologically free, 
there is a point $x_0 \in X$ such that the map $s\mapsto sx_0$ on $\Gamma$ is injective and $\Gamma x_0$
is dense in $X$. It follows that for every two-element subset $\omega = \{g, g'\}$ of $\Omega$ there exists an element 
$h_\omega$ in $\Gamma$ with $h_\omega x_0\in P_\omega$ so that 
$\{K^2h_\omega \}_{\omega\in [\Omega]^2}$ forms a disjoint family.
Write $J$ for the (finite) union of this family.

Let $F$ be any finite subset of $\Gamma$ containing $KJ \cup \{ e \}$. 
This containment is what will allow us to carry the proof of Lemma~\ref{L-zero division}.
Let $U$ be any clopen neighbourhood of $x_0$ small enough so that $\{tU\}_{t\in F}$ is a disjoint family of sets
each of which is contained in some member of $\mathscr{Q}$. Given a $d\in\Nb$, we apply Lemma~\ref{L-free products} to obtain
a clopen tower $(S,C)$ with $|S| = d$ and $SC\subseteq U$. As a consequence of our assumption
that the action has a specification ridge, Lemma~\ref{L-free products}
also gives us a nontrivial involution in $\fA (\Gamma ,X)$ that is freely related to $\fS (S,C)$
and has support in $SC$, and this will get used in Section~\ref{S-theorem}.
Note that the value of $\theta (g,tsx_0 )$ 
for $t\in F$ and $s\in S$ is independent of $s$, so that we have a well-defined map
$\zeta :\Omega\times F\to K$ satisfying
\[
gtsx_0=\zeta (g,t)tsx_0
\]
for all $g\in\Omega$, $t\in F$, and $s\in S$. 

Since the family $\{tU\}_{t\in F}$ is disjoint, the subgroups $\fA (tSt^{-1} ,tC) \subseteq \fA (\Gamma ,X)$ for $t\in F$
are pairwise commuting. 
For each $t\in F$ identify $\fA (tSt^{-1} ,tC)$ with $\fA_S$ in the obvious way. 
Correspondingly, for sets $E\subseteq F$ the subgroup of $\fA (\Gamma ,X)$ generated by the subgroups 
$\fA (tSt^{-1} ,tC) \subseteq \fA (\Gamma ,X)$ for $t\in E$ will be identified
with $\fA_S^E$. Under this identification in the case $E=F$, the conjugate of an element $(g_t)_{t\in F} \in\fA_S^F$ 
by an element in $h\in \fS (F,SC)$ is given by $(g_{\sigma (t)})_{t\in F} \in \fA_S^F$ where $\sigma$
is the permutation of $F$ that corresponds to $h$ in the obvious way. Identifying $\fS (F,SC)$ with $\fS_F$, 
the subgroup of $\fF (\Gamma ,X)$ generated by $\fA_S^F$ and $\fS_F$ can thus be regarded
as the permutational wreath product $\fA_S^F \rtimes \fS_F$. 
The C$^*$-subalgebra of $C^*_\lambda (\fF (\Gamma ,X))$ generated by this copy of $\fA_S^F \rtimes \fS_F$
is the finite-dimensional crossed product $\Cb \fA_S^F \rtimes \fS_F$.
This will not in general be contained $C^*_\lambda (\fA (\Gamma ,X))$,
but the permutation action of $\fS_F$ will be of use
even if we are only interested in $\fA (\Gamma ,X)$ or some proper subgroup of $\fF (\Gamma ,X)$
containing it. 

Let $\sP = \{ p_1 , \dots , p_m \}$ be a collection of pairwise orthogonal projections in $\Cb\fA_S$ summing to $1$. 
Then $C^* (\sP ) \rtimes \fS_F$ is a C$^*$-subalgebra of $\Cb \fA_S^F \rtimes \fS_F$
and can be viewed through a spectral lens as the crossed product $C(\{ 1,\dots , m \}^F)\rtimes \fS_F$
of the action of $\fS_F$ on $\{ 1,\dots , m \}^F$ by permutation of the indices.
The tracial state $\tau$ on $C^*_\lambda (\fA (\Gamma ,X))$ induces a product probability measure
$\nu^F$ on $\{ 1,\dots , m\}^F$. 

Writing
\[
F_K = \bigcap_{s\in K} s^{-1} F \subseteq F,
\]
the following principle and notation will be used repeatedly.

\begin{notation}\label{N-action}
For all $g\in\Omega$ and $A\subseteq_{F_K} \{ 1,\dots , m\}^F$ we have
\[
u_g1_Au_g^*=1_{A'},
\]
where membership of a point $y'$ in the set $A' \subseteq \{ 1,\dots , m\}^F$ is determined by the existence
of a $y\in A$ such that $y'(\zeta(g,t)t)=y(t)$ for all $t\in F_K$. In this situation we will denote the set $A'$
by $gA$, even though $g$ does not actually act on $\{ 1,\dots , m \}^F$. 
\end{notation}

Finally we note that there exists a $\tau$-preserving faithful conditional expectation 
\[
\Eb\colon C^*_\lambda(\dF(G,X))\to C(\{ 1,\dots , m\}^F )
\]
(see Proposition~2.36 in Chapter~V of \cite{Tak79}).

\section{Zero division}\label{S-zero division}

Recall that, to approximate a given noninvertible element $a$ in the group C$^*$-algebra with something invertible,
the strategy is to produce a near block diagonal matrix modelling $a$ that can be unitarily rotated to a nilpotent element.
We cannot do this for $a$ itself but must first unitarily rotate $a$ so that it divides into zero on the left and right
by a common nonzero positive element (in fact projection) out of which we can create enough zeros on the matrix diagonal 
to facilitate the desired rotation to nilpotence. In the crossed product setting of \cite{LiNiu20,BelGefKer25} 
this positive element is a function on the space, and here, analogously, it will be an indicator function
in our embedded permutational Bernoulli structure. In \cite{LiNiu20,BelGefKer25} the positive element
is delivered by Lemma~6.1 and Proposition~6.2 of \cite{LiNiu20}, but we cannot make direct use of these results
or their techniques since they do not presuppose that one is working at a fixed dynamical scale. 
In particular we need to avoid the functional calculus methods used in \cite{LiNiu20}, since 
these are fatal to the control we need on the coefficients of approximations in the algebraic crossed product.
It is the discreteness of our phase space that saves us here.
As in \cite{LiNiu20,BelGefKer25}, Lemma~3.5 of \cite{Ror91} will permit us to assume from the outset,
via a perturbation, that the element $a$ to be approximated 
by something invertible is itself already a two-sided zero divisor, which explains this hypothesis in the lemma below. 

\begin{lemma}\label{L-zero division}
Let $\Gamma\curvearrowright X \subseteq \{ 1,\dots , q \}^\Gamma$ be a right subshift action
that is minimal and topologically free.
Let $G$ be a subgroup of $\fF (\Gamma,X)$ containing $\fA (\Gamma , X)$.
Let $a$ and $b$ be elements of $C_\lambda^*(G)$ such that $b$ is positive and nonzero 
and $ba=0=ab$. Let $\eps>0$.
There exists a $\delta>0$ such that, given any
\begin{itemize}
\item $b_1\in\Cb G$ satisfying $\norm{b-b_1}<\delta$ and 
\item symmetric set $\Omega\Subset G$ strictly containing $L^5$, where $L\coloneqq\supp b_1\cup (\supp b_1 )^{-1} \cup \{ e \}$,
\end{itemize}
if one runs the construction in Section~\ref{S-construction} for $\Omega$ 
to produce a striated clopen tower $(F,S,C)$
and permutational Bernoulli structure $\fS_F \curvearrowright (\{ 1,\dots , m\}^F , \nu^F )$
then, using the notation of Section~\ref{S-construction} 
and writing $\lambda$ for the largest sum of the form $\sum_{i\in I} \nu (\{ i \} )$ for $I\subseteq \{ 1,\dots , m \}$
that does not exceed $1/2$,
there exist a set $D\subseteq F_K$ with cardinality depending only on $b_1$ and $|\Omega |$, a self-adjoint unitary 
$u\in\Cb G$ supported in
\[
\Omega\fA_S^D \Omega\fA_S^D \Omega \fA_S^D \Omega ,
\]
and a set $O\subseteq_D \{ 1,\dots , m\}^F$ with $\nu^F (O) \geq \lambda^{|D|}$
such that
\begin{align*}
\norm{1_Oua}<\eps \hspace{3mm} \text{and} \hspace*{3mm} \norm{au1_O}<\eps .
\end{align*}
\end{lemma}

\begin{proof}
We may assume by normalizing that $\norm{a}=1$ and $\tau (b) = 1$. 
Take a $0 < \delta < 1$ small enough so that $(1-\delta )^{-4}(\norm{b}+\delta)^3\delta < \eps$.
Let $b_1$ and $\Omega$ be as in the lemma statement relative to $\delta$. 
We can write $b_1 = \sum_{g\in L} \beta_g u_g$ for some coefficients $\beta_g$.
We run the construction from Section~\ref{S-construction} for $\Omega$ to obtain 
a striated clopen tower $(F,S,C)$ and permutational Bernoulli structure $\fS_F \curvearrowright (\{ 1,\dots , m\}^F , \nu^F )$,
and we adopt the notation $J$,  $K$, $F_K$, $x_0$, $\zeta$, $h_\omega$, and $\Eb$ from there. 

We construct a subset $Y_0$ of $\{ 1,\dots , m\}^F$ as follows. 
Take an $I\subseteq \{ 1,\dots , m \}$ that gives $\sum_{i\in I} \nu (\{ i \} )$ the largest value not exceeding $1/2$,
and denote this value by $\lambda$.
By the assumption on $\Omega$, there exists an element $h\in \Omega\setminus L^5$.
Since the family $\{K^2h_\omega \}_{\omega\in [\Omega]^2}$ is disjoint 
by the construction in Section~\ref{S-construction},
for every two-element set $\omega= \{ g,g'\} \subseteq \Omega$ 
we have $gh_\omega x_0\neq g'h_\omega x_0$.
Consider the disjoint union
\[
D_0 := \{\zeta (g,h_\omega )h_\omega \}_{\omega = \{ g,g'\}\in [\Omega ]^2}\sqcup\{\zeta (g',h_\omega)h_\omega\}_{\omega = \{ g,g'\} \in [\Omega ]^2} ,
\]
whose cardinality only depends on $b_1$ and $|\Omega|$,
and define $Y_0$ as the set consisting of all points in $\{ 1,\dots , m\}^F$ which take value in $I$ 
at the coordinates in the first of the sets in this union and value in the complement of $I$ at the coordinates in the second one.
Since $KD_0 \subseteq J$ and (by assumption in the construction of Section~\ref{S-construction}) $J \subseteq F$,
by Notation~\ref{N-action} we have projections $1_{gY_0}$ for $g\in\Omega$, and these 
are pairwise orthogonal by our definition of $Y_0$. Note additionally that $\nu^F (Y_0 ) \geq \lambda^{|D_0|}$.
We calculate:
\begin{align}\label{E-cutdown}
b_1 1_{Y_0} b_1^*&=\sum_{g,g' \in L} \beta_g u_g 1_{Y_0} u_{g'}^*\overbar{\beta_{g'}}\\
&=\sum_{g,g'\in L} \beta_g \overbar{\beta_{g'}} 1_{gY_0} u_{g({g'}^{-1})} \notag \\
&=\sum_{s\in L^2}\bigg(\sum_{g\in sL\cap L}\beta_{g}\overbar{\beta_{s^{-1}g}}1_{gY_0}\bigg)u_s . \notag
\end{align}
Set
\[
b_2=\frac{1}{\norm{\Eb((b_11_{Y_0}b_1^*)^2)}^{1/2}}b_1 1_{Y_0}b_1^* ,
\]
which is positive.
Since $\beta_e= \tau (b_1 ) \geq \tau (b) - \delta = 1-\delta$, we have
\begin{equation}\label{expect-bound}
\norm{\Eb((b_11_{Y_0}b_1^*)^2)}=\bigg\| \sum_{g\in L}\bigg(\sum_{s\in gL}|\beta_g|^2|\beta_{s^{-1}g}|^2\bigg)1_{gY_0} \bigg\|
\geq (1-\delta )^4 .
\end{equation}

We pick a $g_0$ that maximizes the quantity
\[
\sum_{s\in gL}|\beta_g|^2|\beta_{s^{-1}g}|^2
\]
over all $g\in L$ and set $O = g_0Y_0$. 
Since $\| \Eb(b_2^2) \|=1$ and the projections $1_{gY_0}$ for $g\in L$ 
are pairwise orthogonal, we then have $\Eb(b_2^2)1_{O} = 1_{O}$.
Define 
\[
D = D_0 \cup \{\zeta(g_0^{-1},t)t:t\in D_0\},
\]
the cardinality of which depends only on $b_1$ and $|\Omega |$.
Then $KD\subseteq K J \subseteq F$ so that $D\subseteq F_K$. We also have
$O\subseteq_D \{ 1,\dots , m\}^F$
and $\nu^F (O) = \nu^F (Y_0 ) = \lambda^{|D_0 |} \geq \lambda^{|D|} $.

Taking the representation \eqref{E-cutdown} of $b_1 1_{Y_0} b_1^*$ as a linear combination of the unitaries associated to $L^2$
with coefficients in $C(\{ 1,\dots , m\}^F )$, squaring it, and normalizing, we obtain an expression for $b_2$
of the form $\sum_{g\in L^4}f_gu_g$ where the $f_g$ are functions in $C(\{ 1,\dots , m\}^F )$.
Using the pairwise orthogonality of the projections $1_{gO} = 1_{gg_0 Y_0}$ for $g\in L^4$, we then
compute that
\[
1_{O}b_2^21_{O}=\sum_{g\in L^4}f_g1_{O}1_{gO}u_g=f_e 1_O = \Eb(b_2^2)1_{O}=1_{O}.
\]

Define $v=u_{hg_0^{-1}}1_{O}b_2$. Then
\[
vv^*= u_h 1_{Y_0} u_h^* = 1_{hY_0}
\]
and
\[
v^*v=b_21_{O}b_2 .
\]
Thus $v$ is a partial isometry, and $v^*v\perp vv^*$ since $hY_0$ is disjoint from the support of all
of the functions $f_g$ appearing in the decomposition of $b_2$. Define
\[
u=1+v+v^*-v^*v-vv^* ,
\]
which is a self-adjoint unitary that conjugates $v^*v$ to $vv^*$. Observe that the support of $u$ is contained in
the union of the sets
\begin{gather*}
h \fA_S^{D} g_0^{-1}L\fA_S^{D} L ,\hspace*{4mm}
L\fA_S^{D} L g_0 \fA_S^{D}  h^{-1} , \\
h\fA_S^{D} h^{-1} , \hspace*{4mm}
L\fA_S^{D} Lg_0 \fA_S^{D} g_0^{-1} L\fA_S^{D} L ,
\end{gather*}
and this union is in turn contained in
\begin{gather*}
\Omega\fA_S^D \Omega\fA_S^D \Omega \fA_S^D \Omega ,
\end{gather*}
which is a subset of $G$ since $\fA_S^D \subseteq \fA (\Gamma ,X) \subseteq G$.
Finally, we observe that
\begin{align*}
\norm{1_{O}ua}
= \| u(u^* 1_O u)a \|
&=\norm{ub_21_{O}b_2a} \\
&\overset{\mathclap{\eqref{expect-bound}}}{\leq} 
(1-\delta )^{-4} \norm{b_1}^3\norm{b_1a} \\
&\leq (1-\delta )^{-4} (\norm{b}+\delta)^3 \delta
< \eps
\end{align*}
and similarly $\norm{au1_{O}} = \norm{ab_21_{O}b_2u}\leq (1-\delta )^{-4}\delta (\norm{b}+\delta)^3 < \eps$, as desired.
\end{proof}

\begin{remark}
Choosing any approximation $a_0$ of the element $a$ above, we get that $1_{O^c}ua_0u1_{O^c}$ approximates $uau$ and is annihilated by $1_O$ on both sides. 
In particular, if a finitely supported $a_0$ is picked at the outset, then we can determine the support of this approximation to $uau$
by including the support of $a_0$ in $\Omega$, as will be done in our 
application in Section~\ref{S-theorem}.
\end{remark}

\section{Producing approximately invariant sets}

The following is a simple refinement of a particular instance of Lemma~5.3 from \cite{KerTuc23}, which has its
origins in Section~4 of \cite{KecTsa08}.

\begin{lemma}\label{L-two sets}
Let $\nu$ be the uniform probability measure on $\{ -1,1 \}$.  
Let $\eps > 0$ and $M\in\Nb$. 
Then there is a $\delta > 0$ such that for every nonempty finite set $F$ of sufficiently large cardinality, 
every $E\subseteq F$ with $|E| \geq (1-\delta )|F|$, and every $\Omega\subseteq\fS_F$ of cardinality $M$,
setting for $k=-1,1$
\[ 
B_k = \big\{ x\in \{ -1,1 \}^F : k\cdot\textstyle\sum_{s\in E} x_s > 0 \big\}
\]
the sets $A_k = \bigcap_{\omega\in\Omega} \omega B_k$ for $k=-1,1$ satisfy the following:
\begin{enumerate}[label=(\roman*)]
\item $\nu^F (A_{-1} ) = \nu^F (A_1 ) \geq \frac12 - \eps$,

\item $\nu^F (\sigma A_k \Delta A_k ) < \eps$ for all $\sigma\in\fS_F$ and $k=-1,1$.
\end{enumerate}
\end{lemma}

\begin{proof}
It is enough to prove the statement for $\Omega = \{ \id \}$, for then we could 
apply it for each $\omega\in\Omega$ taking $E$ to be $\omega E$ (in which case $A_k$ is simply
$\omega B_k$) and $\eps$ to be $\eps /M$ in order to obtain, for all $\sigma\in\fS_F$ and $k=-1,1$,
\begin{align*}
\nu^F (\sigma A_k \Delta A_k ) 
\leq \sum_{\omega\in\Omega} \nu^F (\sigma \omega B_k \Delta \omega B_k ) 
< |\Omega| \cdot \frac{\eps}{M} = \eps 
\end{align*}
and, picking an $\omega_0\in\Omega$,
\begin{align*}
\nu^F (A_k ) 
&\geq \nu^F (\omega_0B_k ) - \sum_{\omega\in\Omega\setminus \{ \omega_0 \}} \nu^F (\omega B_k \Delta \omega_0 B_k ) \\
&\geq \Big( \frac12 - \frac{\eps}{M} \Big) - (|\Omega |-1)\cdot\frac{\eps}{M} \\
&= \frac12 - \eps .
\end{align*}
The fact that $\nu^F (A_{-1} ) = \nu^F (A_1 )$ is clear from the definitions of $B_{-1}$ and $B_1$.

We assume then that $\Omega = \{ \id \}$. For $s\in F$ write $\pi_s$ for the coordinate projection map
$x = (x_t )_{t\in F} \mapsto x_s$ from $\{ -1,1\}^F$ to $\{ -1,1 \}$.
Proceeding as in the proof of Lemma~5.3 in \cite{KerTuc23}, we use Lemma~5.2
of \cite{KerTuc23} to find a $\delta > 0$ such that for every partition $F = F_1\sqcup F_2\sqcup F_3$ of a finite set
with $|F_1| \geq (1-\delta )|F|$ and $|F_2|=|F_3|$ the random variables 
\begin{align}\label{E-rvs}
U = \sum_{s\in F_1} \pi_s, \hspace*{5mm}
V = \sum_{s\in F_2} \pi_s, \hspace*{5mm}
W = \sum_{s\in F_3} \pi_s
\end{align}
on $(\{ -1,1\}^F , \nu^F )$ satisfy $\Pb (W\leq -U \leq V) < \eps$.
Now let $E$ be a subset of a finite set $F$ with $|E|\geq (1-\delta )|F|$, and for $k=-1,1$ set
$A_k = \big\{ x\in \{ -1,1 \}^F : k\cdot\textstyle\sum_{s\in E} x_s > 0 \big\}$. Let $\sigma\in\fS_F$.
Define $U$, $V$, and $W$ as in \eqref{E-rvs} with $F_1 = \sigma E\cap E$, $F_2 = E\setminus\sigma E$, and 
$F_3 = \sigma E\setminus E$. Then $|F_2 | = |F_3 |$ and 
\begin{align*}
|F_1 | = |F| - |F\setminus \sigma E| - |F\setminus E| = |F| - 2|F\setminus E| \geq (1-\delta )|F| ,
\end{align*}
which by our choice of $\delta$ implies, for $k=-1,1$, that
\begin{align*}
\nu^F (\sigma A_k\Delta A_k) 
&= \nu^F (A_k\setminus \sigma A_k )+ \nu^F (\sigma A_k\setminus A_k) \\
&= 2\Pb (W\leq -U \leq V) 
< \eps ,
\end{align*}
yielding (ii). Moreover, for $k=-1,1$ we have
\begin{align*}
\nu^F (A_k )
= \Pb \bigg( \frac{k\cdot\textstyle\sum_{s\in E} \pi_s}{\sqrt{|F|}} > 0\bigg) .
\end{align*} 
The random variables $\pi_s$ for $s\in F$ have common variance $1$,
and so by the central limit theorem
we can make the above probability as close as we wish to $1/2$  
by taking $|F|$ sufficiently large, yielding (i).
\end{proof}

In the following lemma and proof we tacitly identify $(\{ -1,1 \}^d )^F$
with $(\{ -1,1 \}^F )^d$ in the obvious way. Under this identification the measure $(\nu^d )^F$ 
gets expressed as $(\nu^F )^d$.

\begin{lemma}\label{L-product sets}
Let $\nu$ be the uniform probability measure on $\{ -1,1 \}$.   
Let $d\in\Nb$, $\eps > 0$, and $M\in\Nb$. Then, writing $\zeta = \nu^d$,
there is a $\delta > 0$ such that for every nonempty finite set $F$, every $E\subseteq F$ 
satisfying $|E| \geq (1-\delta )|F|$, and every $\Omega\subseteq\fS_F$ of cardinality $M$, 
for $k=-1,1$ setting
\[
B_k = \big\{ x\in \{ -1,1 \}^F : k\cdot\textstyle\sum_{s\in E} x_s > 0 \big\}
\]
and $A_k = \bigcap_{\omega\in\Omega} \omega^{-1} B_k$
the sets $A_\kappa = A_{\kappa (1)} \times\cdots\times A_{\kappa (d)}$
for $\kappa \in \{-1,1 \}^d$
satisfy the following:
\begin{enumerate}[label=(\roman*)]
\item the measures $\zeta^F (A_\kappa )$ for $\kappa \in \{-1,1 \}^d$ 
are equal to some common value greater than $2^{-d} - \eps$,

\item $\zeta^F (\sigma A_\kappa \Delta A_\kappa ) < \eps$ 
for all $\sigma\in\fS_F$ and $\kappa \in \{-1,1 \}^d$,

\item the sets $\bigcup_{\omega\in\Omega} \omega A_\kappa$ for 
$\kappa \in \{-1,1 \}^d$ are pairwise disjoint.
\end{enumerate}
\end{lemma}

\begin{proof}
By Lemma~\ref{L-two sets} we can find a $\delta > 0$ such that, given a nonempty finite set $F$ and an $E\subseteq F$ 
satisfying $|E| \geq (1-\delta )|F|$, for $k=-1,1$ the sets $A_k = \bigcap_{\omega\in\Omega} \omega^{-1} B_k$ 
where
\[
B_k = \big\{ x\in \{ -1,1 \}^F : k\cdot\textstyle\sum_{s\in E} x_s > 0 \big\}
\]
satisfy
\begin{itemize}
\item $\nu^F (A_1 ) = \nu^F (A_2 ) \geq 2^{-1} - \eps /d$, and

\item $\nu^F (\sigma A_k \Delta A_k ) < \eps /d$ for all $\sigma\in\fS_F$ and $k=-1,1$.
\end{itemize}
As in the lemma statement define the sets $A_\kappa$ 
for $\kappa \in \{-1,1 \}^d$ and view them simultaneously as subsets of $(\{ -1,1 \}^d )^F$. 
These sets have the same 
$\zeta^F$-measure since the sets $A_{-1}$ and $A_1$ have the same $\nu^F$-measure and $\zeta^F = (\nu^d )^F$
is the same as the measure $(\nu^F )^d$ under the identification of 
$(\{ -1,1 \}^d )^F$ with $(\{ -1,1 \}^F )^d$,
and this common value is at least $(2^{-1} - \eps /d)^d$, which, assuming $\eps < 1$ as we may, 
is greater than $2^{-d} - \eps$. Thus condition (i) is fulfilled.

For condition (ii), observe that for each $\kappa\in \{-1,1 \}^d$ and $\sigma\in\fS_d$
the set $\sigma A_\kappa \Delta A_\kappa$ is contained in the union of the $d$ sets of the form 
$C_1 \times\cdots\times C_d\subseteq (\{ -1,1\}^F )^d = (\{ -1,1 \}^d )^F$ 
where $C_j = \sigma A_{\kappa (j)} \Delta A_{\kappa (j)}$ for some $1\leq j \leq d$ 
and $C_i = \{ -1,1\}^F$ for $i\neq j$, so that 
\[
\zeta^F (\sigma A_\kappa \Delta A_\kappa ) 
\leq \sum_{i=1}^d \nu^F (\sigma A_{\kappa (i)} \Delta A_{\kappa (i)} )
< d\cdot \frac{\eps}{d} = \eps .
\]

Finally, to verify condition (iii) let $\kappa = (i_1 , \dots , i_d )$ and $\kappa' = (i_1' , \dots , i_d' )$ be distinct elements
of $\{ -1,1 \}^d$. Then there is a $1\leq j\leq d$ such that $i_j \neq i_j'$, in which case
$B_{i_j}$ and $B_{i_j'}$ are disjoint. 
Since for all $\omega\in\Omega$ we have $\omega A_k \subseteq B_k$ for $k=-1,1$, this implies that the sets
$\bigcup_{\omega\in\Omega} \omega A_\kappa 
= \bigcup_{\omega\in\Omega} \omega A_{\kappa (1)} \times\cdots\times \omega A_{\kappa (d)}$ and 
$\bigcup_{\omega\in\Omega} \omega A_{\kappa'}
= \bigcup_{\omega\in\Omega} \omega A_{\kappa' (1)} \times\cdots\times \omega A_{\kappa' (d)}$ 
are disjoint.
\end{proof}

\section{A combinatorial lemma}\label{S-combinatorial}

\begin{lemma}\label{L-combinatorial}
Let $0 < \gamma \leq 1$ and $c\geq 1$. Let $n$ be an integer greater than $(2c+6)/\gamma$.
Let $\sV$ be a disjoint collection of $n^2$ many finite sets with common cardinality greater than $n$. 
Let $O$ be a subset of $X:=\bigsqcup \sV$ of cardinality 
at least $\gamma |X|$. Then, writing $\mu$ for the uniform probability measure on $X$, 
there are a partition $O = O_1 \sqcup O_2 \sqcup O_3$ and
an indexing $\{ V_{i,j} \}_{i,j=1}^n$ of the members of $\sV$ such that, setting 
$V = \bigsqcup_{i=2}^n \bigsqcup_{j=1}^n V_{i,j}$, $V_1 = \bigsqcup_{i=2}^n V_{i,1}$, and $P = \bigsqcup_{j=1}^n V_{1,j}$, 
\begin{enumerate}[label=(\roman*)]
\item $\mu (O_1 \cap \bigsqcup_{j=2}^n V_{i,j} ) \geq \mu (V_{i,1} ) + c/n^2$ for every $i=2,\dots , n$,

\item $\mu ((O_2 \cap V)\setminus V_1 ) \geq \mu (P) + c/n^2$, 

\item $\mu (O_3 \cap P) \geq c/n^2$.
\end{enumerate}
\end{lemma}

\begin{proof}
Take an indexing $\{ V_{i,j} \}_{i,j=1}^n$ of the members of $\sV$ such that 
$(i,j) \mapsto \lambda_{i,j} := \mu (O\cap V_{j,i} )$ 
is a nonincreasing function with respect to the lexicographic order on $\{ 1,\dots , n \}^2$.
Set $O_3 = O \cap P$.
For each $2\leq i\leq n$ and $1\leq j\leq n$ choose a partition $O_{i,j} \sqcup O_{i,j}'$ of $O\cap V_{i,j}$
into two sets whose cardinality differs by at most $1$. Set $O_1 = \bigsqcup_{i=2}^n \bigsqcup_{j=1}^n O_{i,j}$
and $O_2 = \bigsqcup_{i=2}^n \bigsqcup_{j=1}^n O_{i,j}'$. Then we have $O = O_1 \sqcup O_2 \sqcup O_3$.

Set $\lambda_j = \sum_{i=1}^n \lambda_{i,j}$.
Then $\lambda_1 \geq \lambda_2 \geq\dots\geq \lambda_n$ and $\sum_{j=1}^n \lambda_i \geq \gamma$. Combining these, we get 
\begin{equation}\label{lambda1}
\lambda_1 \geq \frac{\gamma}{n}.
\end{equation}
For every $j=1,\dots , n$ we have $\lambda_{i-1,n} \geq \lambda_{i,j}$ for $i=2,\dots , n$ so that
\[
\lambda_n 
\geq \sum_{i=2}^n \lambda_{i,j}
= \lambda_j - \lambda_{1,j}
\geq \lambda_j - \frac{1}{n^2}
\]
and hence
\begin{align}\label{E-lambdan}
\lambda_n 
\geq \frac{1}{n}\sum_{j=1}^n \Big( \lambda_j - \frac{1}{n^2} \Big) 
\geq \frac{\gamma n-1}{n^2}.
\end{align}

For every $i=1,\dots , n$ the difference between the cardinalities of $O_1 \cap \bigsqcup_{j=1}^n V_{i,j}$
and $O_2 \cap \bigsqcup_{j=1}^n V_{i,j}$ is at most $n$, and so the $\mu$-measures of these sets differ
by at most $n/|X|$, which is less than $1/n^2$ since the cardinality of the sets $V_{i,j}$ is greater than $n$
by hypothesis. From (\ref{E-lambdan}) we therefore obtain, for $i=2,\dots , n$ and $k=1,2$,
\[
\mu \Big( O_k \cap \bigsqcup_{j=1}^n V_{i,j} \Big) \geq \frac{\gamma n-2}{2n^2}.
\]
Utilizing our assumption that $n\geq (2c+6)/\gamma$, it follows that
\begin{align*}
\mu \Big(O_1 \cap \bigsqcup_{j=2}^n V_{i,j} \Big) 
&\geq \mu \Big( O_1 \cap \bigsqcup_{j=1}^n V_{i,j} \Big) - \mu (V_{i,1} ) \\
&\geq \frac{\gamma n-2}{2n^2} - \frac{1}{n^2} \\\
&= \frac{1}{n^2} + \frac{\gamma n-6}{2n^2} \\
&\geq \mu (V_{i,1} ) +  \frac{c}{n^2} .
\end{align*}
and
\begin{align*}
\mu ((O_2 \cap V)\setminus V_1 ) 
&\geq \sum_{i=2}^{n} \mu \Big(O_2\cap \bigsqcup_{j=1}^n V_{i,j} \Big) - \sum_{i=2}^n \mu (V_{i,1} ) \\
&\geq (n-1) \cdot\frac{\gamma n-2}{2n^2} - (n-1) \cdot\frac{1}{n^2} \\
&= \frac{(n-1)(\gamma n -4)}{2n^2}\\
&\geq \frac{(n-1)(c+1)}{n^2}\\
&\geq \frac{n + c}{n^2} \\
&= \mu (P) + \frac{c}{n^2} .
\end{align*}
Finally we observe by \eqref{lambda1} and, again, our lower bound on $n$ that
\[
\mu (O_3 \cap P) = \lambda_1 \geq \frac{\gamma}{n} \geq \frac{2c+6}{n^2} >\frac{c}{n^2} .\qedhere
\]
\end{proof}

\section{Proof of Theorem~\ref{T-stable rank one}}\label{S-theorem}

With the lemmas of the last five sections at hand,
we now embark on the proof of Theorem~\ref{T-stable rank one}.
So let $G$ be a subgroup of $\fF (\Gamma,X)$ containing $\fA (\Gamma , X)$,
let $a$ be a noninvertible element of $C^*_\lambda (G)$, and let $\eps > 0$, and let 
us show the existence of an invertible element $\tilde{a} \in C^*_\lambda (G)$
such that $ \| a - \tilde{a} \| < \eps$. We know by Remark~\ref{R-simple} that $C^*_\lambda (G)$ is simple,
and so Lemma~3.5 in \cite{Ror91} permits us to assume that there exists a nonzero positive 
element $b\in C^*_\lambda (G)$ such that $ab = ba = 0$ (this is the first of two places in the proof
where we use the hypothesis of having a specification ridge).

Let $\delta > 0$ be as given by Lemma~\ref{L-zero division} for $\eps/12$ and with respect to $a$ and $b$.
Take $a_0 , b_0 \in \Cb G$ such that $\| a - a_0 \| < \eps /4$
and $\| b - b_0 \| < \delta$.
Then there is an $L\Subset G$ such that we can write
$a_0 = \sum_{s\in L} \alpha_s u_s$ and $b_0 = \sum_{s\in L} \beta_s u_s$ for some scalars $\alpha_s$ and $\beta_s$.
We may assume $L$ to be symmetric and contain $e$.
Take an $h\in G \setminus L^5$ and set $\Omega = L^5 \sqcup \{ h,h^{-1} \}$.
We now run the construction in Section~\ref{S-construction} for $\Omega$ to get a striated clopen tower
$(F,S,C)$ and the associated permutational Bernoulli structure $\fS_F \curvearrowright (\{ 1,\dots , m\}^F , \nu^F )$.
We adopt the notation of Section~\ref{S-construction}, so that $K$ denotes the set $K_0 \cup K_0^{-1} 
\cup \{ e \}$ where $K_0$ is the set
all of elements in $\Gamma$ that locally implement the action of elements of $\Omega$ on members of clopen partitions, 
$F_K$ is the set $\bigcap_{s\in K} s^{-1} F \subseteq F$, 
and $\fA_S$ is abusively used to denote each of the embedded copies of itself in the permutational structure.
We may take $C$ small enough so that each element of $\Omega L \Omega$ acts as a single
element of $\Gamma$ on each level of the tower $(FS,C)$.
The construction in Section~\ref{S-construction} affords us the freedom to take $S$ to be as large in cardinality and
$F$ to be as left invariant as we wish, and also to choose the base $(\{ 1,\dots , m\}^F , \nu^F )$ of the Bernoulli structure, within the limits imposed by the size of $S$. 
We will now proceed to specify all of these parameters.

In view of our choice of $\delta$, 
and writing $\lambda$ for the largest sum of the form $\sum_{i\in I} \nu (\{ i \} )$ for $I\subseteq \{ 1,\dots , m \}$
that does not exceed $1/2$, 
Lemma~\ref{L-zero division} gives us a set $D\subseteq F_K$ with cardinality 
depending only on $b_0$ and $a_0$, a self-adjoint unitary $u\in\Cb G$ supported in
$\Omega \fA_S^D \Omega \fA_S^{D} \Omega\fA_S^{D} \Omega$,
and a cylinder set $O\subseteq_D \{ 1,\dots , m\}^F$ with $\nu^F (O)\geq \lambda^{-|D|}$ such that 
\begin{align}\label{E-O}
\norm{1_Oua}<\frac{\eps}{12} \hspace*{3mm} \text{and}\hspace*{3mm}\norm{au1_O}< \frac{\eps}{12} .
\end{align}
Set $\theta = 3^{-|D|}$.
Take an even integer $d > 0$ such that the integer $n := 2^{d/2}$ is larger than $22/\theta$ (in particular, $\theta /2 \geq 1/(2n^2)$).
Note that $\theta$, $d$, and $n$ do not depend on $(F,S,C)$ or the permutational Bernoulli structure,
which we therefore have the freedom to chose without affecting these quantities. 

Let $\delta'>0$ be as given by Lemma~\ref{L-product sets} with respect to 
the parameters $d$, $1/(2n^4|K|)$, and $|K|^{21}$, respectively.
Since $\Gamma$ is amenable and infinite and the cardinality of $D$ does not depend on the choice of $F$, we can take $F$ to be large enough to satisfy
\begin{equation}\label{E-F}
2^{|F|}\Big(1-\frac{\theta}{2} \Big)>n
\end{equation} 
and  sufficiently left invariant so that there exists a set $E\subseteq F$ satisfying $|E| \geq (1-\delta')|F|$ and
\begin{align}\label{E-ED}
K^{21} E \subseteq F\setminus D .
\end{align}

Now we specify $S$ and the permutational Bernoulli structure. 
These will depend on the cardinality of $F$.
By the representation theory of alternating groups \cite{FulHar91,JamKer81}, 
the Wedderburn-Artin decomposition of the group ring $\Cb\fA_S$
has the form $\Cb \oplus (\bigoplus_{i\in I} M_{k_i} )$
where $k_i \geq |S| - 1$ for every $i\in I$ assuming $|S| \geq 7$.
Write $\trs$ for the canonical tracial state on $\Cb\fA_S$, i.e. the one arising via the left regular representation of $\fA_S$. Then, for all 
$w = (z, (w_i )_{i\in I} ) \in \Cb \oplus (\bigoplus_{i\in I} M_{k_i} ) = \Cb\fA_S$ one has
\[
\trs (w) = \frac{z}{|\fA_S|} + \sum_{i\in I} \frac{k_i^2}{|\fA_S|} \trs_i (w_i ),
\]
where $\trs_i$ is the tracial state on $M_{k_i}$.

Take $S$ large enough so that
\[
\frac{2^{d+1}}{|S|-1}<\left(1-\frac{1}{2n^2}\right)^{1/|F|}.
\]
This allows us to find projections $p_{z,i}\in M_{k_i}$ for all $i\in I$ and $z\in\{-1,1\}^d$ such that $\trs_i(p_{z,i})=\trs_i(p_{z',i})$ for all $z,z'\in\{-1,1\}^d$ and, writing $p_{*,i}\coloneqq 1_{M_{k_i}}-\sum_{z\in\{-1,1\}^d}p_{z,i}$, 
\[
\trs_1(p_{*,i})<\frac12 \bigg(1-\frac{1}{2n^2}\bigg)^{1/|F|} .
\] 
We set $p_z = (0,(p_{z,i})) \in \Cb\fA_S$ for all $z\in\{-1,1\}^d$, and $p_* = (1,(p_{*,i}))\in \Cb\fA_S$. 
Then $\{p_z:z\in\{-1,1\}^d\}$ is a collection of pairwise orthogonal projections of the same trace that
sum to $1-p_*$, where $\trs(p_*)<(1-1/(2n^2))^{1/|F|}$, and that are pairwise Murray--von Neumann equivalent
in $\Cb \fA_S$.

Using these projections to generate a commutative C$^*$-subalgebra 
$N\subseteq \Cb \fA_S \subseteq C^*_\lambda(G)$, 
we obtain what we will take as the base space $\{1,\dots,m\}$ in the construction in Section~\ref{S-construction},
namely $Z\coloneqq \{-1,1\}^d\sqcup\{*\}$ with the measure $\nu$ induced by the canonical trace, which is 
uniformly distributed over $\{-1,1\}^d$ with $\nu(\{ * \})<(1-1/(2n^2))^{1/|F|}$. Therefore, writing $Q_1=(\{-1,1\}^d)^F\subseteq Z^F$ and $\mu=\nu^F$ (the measure induced by the canonical trace on $\Cb\fA_S^F \subseteq C_\lambda^*(\fA(\Gamma,X)$), we have
\[
\mu(Q_1)=\nu(\{-1,1\}^d)^{|F|}\geq 1-\frac{1}{2n^2}.
\]

Consider the element $a_1 := \unit_{O^c} u a_0 u \unit_{O^c}$. We can write
$a_1 = \sum_{s\in L_1} \rho_s u_s$ for some scalars $\rho_s$ where (recalling that 
$u$ is supported in $\Omega \fA_S^{D} \Omega \fA_S^{D} \Omega \fA_S^{D} \Omega$ and that $1_O$, and hence also $1_{O^\comp}$,
is supported in $\fA_S^D$)
\begin{align}
L_1
= (\fA_S^D \Omega )^4 L (\Omega \fA_S^D )^4  \Subset G .
\end{align}
Also, recalling that $u$ is self-adjoint and using \eqref{E-O},
\begin{align*}
\| u a u^* - a_1 \| 
&= \| uau - a_1 \| \\
&\leq \|uau-1_{O^c} u a u 1_{O^c}\|+\|a-a_0\| \\
&= \|1_O u a u 1_O + u a u 1_O + 1_O u a u\| + \|a-a_0\| \\
&< \frac{\eps}{4} + \| \unit_O u a u \| + 2 \| u a u \unit_O \| \\
&\leq \frac{\eps}{4} + \frac{\eps}{12} + 2\cdot \frac{\eps}{12}
= \frac{\eps}{2} .
\end{align*}
Writing $\sI$ and $\sU$ for the groups
of invertible and unitary elements in $C^*_\lambda (G)$, respectively, this yields
\begin{align*}
\inf_{\tilde{a}\in \sI} \| a - \tilde{a} \| 
&= \inf_{\tilde{a}\in \sI} \inf_{u\in\sU} \| a - u^* \tilde{a} u \| \\
&= \inf_{\tilde{a}\in \sI} \inf_{u\in\sU} \| uau^* - \tilde{a} \| \\
&\leq \inf_{\tilde{a}\in \sI} \inf_{u\in\sU} (\| uau^* - a_1 \| + \| a_1 - \tilde{a} \|) \\
&< \frac{\eps}{2} + \inf_{\tilde{a}\in \sI} \| a_1 - \tilde{a} \| 
\end{align*}
It thus now suffices to show that there exists an invertible element $\tilde{a}$ in 
$C^*_\lambda (G)$ such that $\| a_1 - \tilde{a} \| \leq \eps /2$.
We will do this by constructing unitaries $u_1 , u_2 \in C^*_\lambda (G)$ such that $u_1 a_1 u_2$ is nilpotent, which
is sufficient since then $u_1 a_1 u_2 + \eps /2$ is invertible, which in turn implies the invertibility of the element
$a_1 + (\eps /2)u_1^* u_2^* = u_1^* (u_1 a_1 u_2 + \eps /2) u_2^*$, whose norm distance to $a_1$
is at most $\eps /2$. The virtue of working with $a_1$ is that it satisfies the genuine zero division 
\[
a_1 1_O = 1_O a_1 = 0 .
\]

Recall that $K = K_0 \cup K_0^{-1} \cup \{ e \}$ where $K_0$ is the set
all of elements in $\Gamma$ that locally implement the action of elements of $\Omega$ on members of clopen partitions. 
By the definition of $F_K$, for each $s\in K$ the conjugates by $u_s$ of the indicator functions of cylinder sets in $Z^F$ 
supported on $F_K$ are again indicator functions of cylinder sets in $Z^F$ (but no longer necessarily supported on $F_K$)
determined by a partial permutation of $F$ with domain $F_K$, which we extend in some arbitrary
way to a full permutation $\sigma_s$ of $F$. Write $\sK$ for the set of all $\sigma_s$ for $s\in K$.
Note that the cardinality of $\sK$ depends only on $b_0$, which will allow us to apply Lemma~\ref{L-product sets} later.

We identify $Q_1=(\{ -1,1 \}^d )^F$ with $(\{ -1,1 \}^F )^d$ in the obvious way and equip it with the uniform probability measure $\mu_1$.
By our choice of $\delta'$ as supplied by Lemma~\ref{L-product sets},
if for $k=-1,1$ we set
\[
B_k = \big\{ x\in \{ -1,1 \}^F : k\cdot\textstyle\sum_{s\in E} x_s > 0 \big\}
\]
and $A_k = \bigcap_{\omega\in\sK^{21}} \omega^{-1} B_k$,
then the sets $A_\kappa = A_{\kappa (1)} \times\cdots\times A_{\kappa (d)}$
for $\kappa \in \{-1,1 \}^d$
satisfy the following:
\begin{enumerate}
\item\label{I-A1} the measures $\mu_1 (A_\kappa )$ for $\kappa \in \{-1,1 \}^d$ 
are equal to some common value greater than $2^{-d} (1- \theta /2)$,

\item $\mu_1 (\sigma A_\kappa \Delta A_\kappa ) < 1/(2n^4 |K|^{21})$
for all $\sigma\in\fS_F$ and $\kappa \in \{-1,1 \}^d$,

\item the sets $\bigcup_{\omega\in\sK^{21}} \omega A_\kappa$ for 
$\kappa \in \{-1,1 \}^d$ are pairwise disjoint.
\end{enumerate}
Viewing every $A_\kappa$ as a subset of $Z^F$, define $\tilde{A}_\kappa = \pi_E^{-1}(\pi_E(A_\kappa))\subseteq_E Z^F$. Using the above we have

\begin{enumerate}[resume]
\item\label{I-A2} $\mu (\sigma \tilde{A}_\kappa \Delta \tilde{A}_\kappa ) < 1/(2n^4 |K|^{21})$ 
for all $\sigma\in\fS_F$ and $\kappa \in \{-1,1 \}^d$,

\item\label{I-A3} the sets $\bigcup_{\omega\in\sK^{21}} \omega \tilde{A}_\kappa$ for 
$\kappa \in \{-1,1 \}^d$ are pairwise disjoint.
\end{enumerate}
By \eqref{I-A2}, for every $\kappa$ we have, using the fact that $|\sK^{21}| \leq |K|^{21}$,
\begin{align}\label{E-core}
\mu \bigg(\bigcap_{s\in \sK^{21}} s^{-1} \tilde{A}_\kappa \bigg) 
&\geq \mu (\tilde{A}_\kappa ) - \sum_{\sigma\in\sK^{21}} \mu (\sigma \tilde{A}_\kappa \Delta \tilde{A}_\kappa ) \\
&\geq \mu (\tilde{A}_\kappa ) - |K|^{21} \cdot \frac{1}{2n^4 |K|^{21}} \notag \\
&= \mu (\tilde{A}_\kappa ) - \frac{1}{2n^4} . \notag
\end{align}

Define $Q_2 = \bigsqcup_{\kappa\in\{ -1,1\}^d} A_\kappa$.
Write $\mu_2$ for the uniform probability measure on $Q_2$, in which case $\mu_2 (\cdot ) = \mu (Q_2)^{-1} \mu (\cdot )$. 
By \eqref{I-A1} we have
\begin{equation}\label{E-Q2}
\mu (Z^F \setminus Q_2 ) 
= 1- \mu \bigg(\bigsqcup_{\kappa\in\{ -1,1 \}^d} A_\kappa \bigg)\leq 1-\mu_1\bigg(\bigsqcup_{\kappa\in\{-1,1\}^d} A_\kappa\bigg)
\leq \frac{\theta}{2} .
\end{equation}

Observe that the lower bound on $\mu (O)$ that comes from our application of Lemma~\ref{L-zero division}
will be at least $\theta$, since we defined $\theta$ to be $3^{-|D|}$. Recall that $\theta$ depends only on $b_1$. Set $O' = O\cap Q_2$. 
By \eqref{E-Q2} we have 
\begin{align*}
\mu (O' ) 
\geq \mu (O) - \mu (Z^F \setminus Q_2 ) 
\geq \mu (O) - \frac{\theta}{2}
\geq \frac{\theta}{2} .
\end{align*}
Given that $\mu_2$ is the uniform probability measure on $Q_2$ and the sets $A_\kappa$ for $\kappa\in\{ -1,1 \}^d$ 
all have the same $\mu_2$-measure (and thus the same cardinality, which is greater than $n$ by \eqref{I-A1} and \eqref{E-F}), we can then apply
Lemma~\ref{L-combinatorial} using the sets $A_\kappa=\tilde{A}_\kappa\cap Q_1 = \tilde{A}_\kappa\cap Q_2$ 
in order to get a partition $O' = O_1 \sqcup O_2 \sqcup O_3$ and an indexing $\{ V_{i,j} \}_{i,j=1}^n$
of the sets $\tilde{A}_\kappa$ for $\kappa\in \{ -1,1 \}^d$ such that, setting 
\begin{gather*}
V = \bigsqcup_{i=2}^n \bigsqcup_{j=1}^n V_{i,j},\hspace*{4mm}
V_1 = \bigsqcup_{i=2}^n V_{i,1},\hspace*{4mm}
V_2 = V\setminus V_1,\\ 
U_i = \bigsqcup_{j=2}^n V_{i,j},\hspace*{4mm}
P = \bigsqcup_{j=1}^n V_{1,j}
\end{gather*}
we have (using the fact that $V_{i,j}\cap Q_2 = V_{i,j}\cap Q_1$ for all $i,j$)
\begin{enumerate}[resume]
\item\label{I-O1} $\mu_2 (O_1 \cap U_i )
\geq \mu_2 (V_{i,1}\cap Q_1) + \mu (Q_2 )^{-1} n^{-2}$ for every $i=2,\dots , n$,

\item\label{I-O2} $\mu_2 (O_2 \cap V_2 ) > \mu_2 (P\cap Q_1) + \mu (Q_2 )^{-1} n^{-2}$,

\item\label{I-O3} $\mu_2 (O_3 \cap P)\geq \mu (Q_2 )^{-1} n^{-2}$.
\end{enumerate}
 
We define the remainder and $\sK^{21}$-boundary for the set $V$ by
\begin{align}\label{E-RB}
R = Z^F \setminus V \hspace*{5mm} \text{and} \hspace*{5mm} B = V \setminus \bigcap_{s\in \sK^{21}} s^{-1} V . 
\end{align}
From \eqref{E-core} we get
\begin{align}\label{E-B}
\mu (B) 
\leq \mu \bigg( \bigsqcup_{i=2}^n \bigsqcup_{j=1}^n \bigg(V_{i,j} \setminus \bigcap_{s\in \sK^{21}} s^{-1} V_{i,j} \bigg) \bigg) 
&= \sum_{i=2}^n \sum_{j=1}^n \mu \bigg(V_{i,j} \setminus \bigcap_{s\in \sK^{21}} s^{-1} V_{i,j} \bigg) \\
&\leq n(n-1) \cdot \frac{1}{2n^4} 
< \frac{1}{2n^2} . \notag
\end{align}
We claim that
\begin{enumerate}[resume]
\item\label{I-O12} $\mu ((O_1 \cap U_i) \setminus B) > \mu (V_{i,1} )$ for every $i=2,\dots , n$,

\item\label{I-O22} $\mu ((O_2 \cap V_2 )\setminus B) > \mu (R)$,

\item\label{I-O32} $\mu (O_3 \cap R) > \mu (B)$. 
\end{enumerate}
Since $P\subseteq R$, the inequality \eqref{I-O32} follows from \eqref{E-B} and \eqref{I-O3},
recalling that $\mu_2 (\cdot ) = \mu (Q_2)^{-1} \mu (\cdot )$.
Using \eqref{I-O1}, \eqref{E-B}, and the fact that $\mu (Z^F\setminus Q_1 ) \leq 1/(2n)^2$, we get
\begin{align*}
\mu ((O_1 \cap U_i )\setminus B)
&\geq \mu (O_1 \cap U_i ) - \mu (B) \\
&\geq \mu (V_{i,1} \cap Q_1 ) + \frac{1}{n^2} - \frac{1}{2n^2} \\
&\geq \mu (V_{i,1} ) - \mu (Z^F\setminus Q_1 ) + \frac{1}{2n^2} \\
&\geq \mu (V_{i,1} ) ,
\end{align*}
yielding \eqref{I-O12}.
Again using \eqref{E-B} and the fact that $\mu (Z^F\setminus Q_1 ) \leq 1/(2n)^2$, this time along with \eqref{I-O2}, we get
\begin{align*}
\mu ((O_2 \cap V_2 )\setminus B)
&\geq \mu (O_2 \cap V_2 ) - \mu (B) \\
&> \mu (P \cap Q_1 ) + \frac{1}{n^2} - \frac{1}{2n^2} \\
&= \mu (R \cap Q_1 ) + \frac{1}{2n^2} \\
&\geq \mu (R) - \mu (Z^F \setminus Q_1 ) + \frac{1}{2n^2} \\
&\geq \mu (R) ,
\end{align*}
yielding \eqref{I-O22}.

We will next need to invoke strict comparison in a suitable C$^*$-subalgebra of $C^*_\lambda (G)$
in order to implement some Murray--von Neumann subequivalences.
For this we appeal to results of Ozawa from \cite{Oza25} involving C$^*$-selflessness.
The construction in Section~\ref{S-construction} from which our striated tower $(F,S,C)$ was produced
involved the use of Lemma~\ref{L-free products}, which, in addition to yielding the sets $S$ and $C$, 
supplies us with a nontrivial involution in $\fA (\Gamma ,X)$ freely related to $\fA (S,C)$ and
with support contained in $SC$ (this is the second place in the proof where the hypothesis of having a specification ridge 
is being used).
By conjugating this involution with the elements of $F$
we obtain, for each copy of $\fA_S$ in the embedded direct product $\fA_S^F \subseteq \fA (\Gamma ,X)$,
a nontrivial involution in $\fA (\Gamma ,X)$ 
freely related to it and with support contained in the underlying clopen tower.
This means that $\Cb \fA_S^F$ is contained in a C$^*$-subalgebra of 
$C^*_\lambda (\fA (\Gamma ,X))$ of the form $C^*_\lambda ((\fA_S *\Zb_2 )^F )$.
As pointed out after the statement of Theorem~1 in \cite{Oza25}, nonelementary free products 
admit topologically free minimal extreme boundary actions by \cite{FimLeMMooSta22}, and
so by Proposition~15 of \cite{Oza25} these groups satisfy property ${\rm P}_{{\rm PHP}}$ as defined at the
beginning of Section~8 therein. Ozawa moreover observes
that property ${\rm P}_{{\rm PHP}}$ is closed under direct products, and so $(\fA_S *\Zb_2 )^F$ has this
property.
It follows by Theorem~14 of \cite{Oza25} that $C^*_\lambda ((\fA_S *\Zb_2 )^F )$ is selfless
with respect to the canonical tracial state, which is unique (as follows from selflessless \cite[Theorem~3.1]{Rob25} or by 
much earlier results \cite{PasSal79}), and so this C$^*$-algebra has strict comparison by Theorem~3.1 of \cite{Rob25}. The upshot for us is:
\begin{enumerate}[resume]
\item\label{I-comparison} $\Cb \fA_S^F$ is contained in a C$^*$-subalgebra of $C^*_\lambda (G)$
that has strict comparison and a unique tracial state (the restriction of $\tau$).
\end{enumerate}

We define
\begin{align*}
Y_1 = (O_1 \cap V_2)\setminus B , \quad
Y_2 = (O_2 \cap V_2)\setminus B , \quad
Y_3 = O_3 \cap R .
\end{align*}
Since the collections $\{ V_{i,1} \}_{i=2}^{n}$ and $\{ \bigsqcup_{j=2}^n V_{i,j} \}_{i=2}^{n}$ are both disjoint,
by \eqref{I-comparison} and \eqref{I-O12} there exists a partial isometry $z\in C^*_\lambda (G)$ such that 
$z^* z = \unit_{V_1}$ and $zz^* \leq \unit_{Y_1}$ (recall that Cuntz subequivalence is the same as 
Murray--von Neumann subequivalence when restricting to projections, as shown in Section~2 of \cite{Ror92}).
Since $z^* z$ and $zz^*$ are orthogonal projections, the element $u:= (1-z^* z - zz^* ) + z + z^*$
is self-adjoint and unitary, with
\begin{itemize}
\item $u\unit_{V_1} u \leq \unit_{Y_1}$, and in particular $u\unit_{V_1} = \unit_{Y_1} u\unit_{V_1}$,

\item $u\unit_{(V_1 \sqcup Y_1)^\comp} = \unit_{(V_1 \sqcup Y_1 )^\comp}$.
\end{itemize}
Note that by \eqref{I-O12} we can assume that $u$ is in fact implemented \say{row by row}, so that
\begin{itemize}
\item $u1_{V_{i,1}}u\leq 1_{Y_1\cap U_i}$.
\end{itemize}
We similarly use \eqref{I-comparison}, now in conjunction with \eqref{I-O22}) and \eqref{I-O32}, respectively,
to obtain self-adjoint unitaries $v,w\in C^*_\lambda (G)$ such that
\begin{itemize}
\item $v\unit_R v \leq \unit_{Y_2}$, and in particular $v\unit_R = \unit_{Y_2} v\unit_R$,

\item $v\unit_{(R \sqcup Y_2)^\comp} = \unit_{(R \sqcup Y_2 )^\comp}$,
\item $w\unit_B w \leq \unit_{Y_3}$, and in particular $w\unit_B = \unit_{Y_3} w\unit_B$,

\item $w\unit_{(B \sqcup Y_3)^\comp} = \unit_{(B \sqcup Y_3 )^\comp}$.
\end{itemize}
Define $b = uwa_1uv$. We will now construct unitaries $z_1 , z_2 \in \Cb \fA_S^F$
such that $z_1 b z_2$ is nilpotent.

Partition the unit in $C^*_\lambda (G)$ as $\unit_R + \sum_{i=2}^{n} \sum_{j=1}^n \unit_{V_{i,j}}$.
We can correspondingly view elements of $C^*_\lambda (G)$ as $(1 + n(n-1)) \times (1 + n(n-1))$ matrices.
Let us verify in this matrix representation that $b$ takes the following form (illustrated in the case $n=4$), 
where the first row and column correspond to $R$
and the remaining rows and columns correspond to the sets $V_{i,j}$ under the lexicographic order on their indices:
\begin{gather*}
\mbox{\footnotesize 
$\left[
\begin{array}{c | c c c c | c c c c | c c c c}
0 & 0 & \entry & \entry & \entry & 0 & \entry & \entry & \entry & 0 & \entry & \entry & \entry \\
\hline
0 & 0 & 0 & 0 & 0 & 0 & 0 & 0 & 0 & 0 & 0 & 0 & 0\\
0 & 0 & \entry & \entry & \entry & 0 & 0 & 0 & 0 & 0 & 0 & 0 & 0\\
0 & 0 & \entry & \entry & \entry & 0 & 0 & 0 & 0 & 0 & 0 & 0 & 0\\
0 & 0 & \entry & \entry & \entry & 0 & 0 & 0 & 0 & 0 & 0 & 0 & 0\\
\hline
0 & 0 & 0 & 0 & 0 & 0 & 0 & 0 & 0 & 0 & 0 & 0 & 0\\
0 & 0 & 0 & 0 & 0 & 0 & \entry & \entry & \entry & 0 & 0 & 0 & 0\\
0 & 0 & 0 & 0 & 0 & 0 & \entry & \entry & \entry & 0 & 0 & 0 & 0\\
0 & 0 & 0 & 0 & 0 & 0 & \entry & \entry & \entry & 0 & 0 & 0 & 0\\
\hline
0 & 0 & 0 & 0 & 0 & 0 & 0 & 0 & 0 & 0 & 0 & 0 & 0\\
0 & 0 & 0 & 0 & 0 & 0 & 0 & 0 & 0 & 0 & \entry & \entry & \entry\\
0 & 0 & 0 & 0 & 0 & 0 & 0 & 0 & 0 & 0 & \entry & \entry & \entry\\
0 & 0 & 0 & 0 & 0 & 0 & 0 & 0 & 0 & 0 & \entry & \entry & \entry
\end{array}
\right]$} .
\end{gather*}
Since $Y_1 , Y_2 \subseteq O$, we have 
\begin{itemize}
\item $a_1uv\unit_R = a_1u\unit_{Y_2} v \unit_R = a_1\unit_{Y_2} v \unit_R = 0$, giving the first column of zeros,

\item $a_1uv\unit_{V_1} = a_1u\unit_{V_1} = a_1\unit_{Y_1} u\unit_{V_1} = 0$, giving the other $n-1$ columns of zeros,

\item $\unit_{V_1} uwa_1 = \unit_{V_1} u \unit_{Y_1} wa_1 = \unit_{V_1} u \unit_{Y_1} a_1 = 0$, giving the $n-1$ rows of zeros.
\end{itemize}

To confirm the other zero entries, we first verify that, given any subcollection $\sW$ of the sets $V_{i,j}$
and writing $W := \bigsqcup \sW$, we have:
\begin{enumerate}[resume]
\item\label{I-factone} $a_1\unit_W = \unit_{W\sqcup R} a_1 \unit_W$,

\item\label{I-facttwo} $a_1\unit_R= \unit_{R\sqcup B} a_1 \unit_R$.
\end{enumerate}
For this we will use Notation~\ref{N-action} along with the following simple observation:
\begin{enumerate}[resume]
\item\label{I-obstwo} Let $F_0$ and $F_1$ be disjoint subsets of $F$, let $A \subseteq_{F_0} Z^F$, 
and let $g$ be an element of the subgroup $\fA_S^{F_1}$ of $\fA_S^F \subseteq \fA (\Gamma ,X)$. 
Then
\begin{align*}
u_g \unit_A u_g^* = \unit_A .
\end{align*}
\end{enumerate}
Recall that $a_1=\sum_{s\in L_1} \rho_s u_s$ where $L_1 = (\fA_S^D \Omega )^4 L (\Omega\fA_S^D )^4$.
Recall also 
that each element of $\Omega L \Omega$ acts as a single
element of $\Gamma$ on each level of the tower $(FS,C)$.

Given an $s\in L_1$, we can write it as $s_1 t_1 s_2 t_2 \cdots s_7 t_7 s_8$ where $s_1 , \dots , s_8\in \fA_S^D$ 
and $t_1 , \dots , t_7\in \Omega\cup \Omega L \Omega \subseteq \Omega^3$.
Since $E\subseteq F\setminus D$ and $W\subseteq_E Z^F$, 
by \eqref{I-obstwo} we have $u_{s_8}\unit_W u_{s_8}^* = \unit_W$.
Since $K^3 E \subseteq K^{21} E \subseteq F\setminus D$ by \eqref{E-ED}, using Notation~\ref{N-action} 
we can write $u_{t_7} \unit_W u_{t_7}^* = \unit_{t_7 W}$.
Again using that $K^3 E\subseteq K^{21} E \subseteq F\setminus D$, we then have by \eqref{I-obstwo} 
that $u_{s_7}\unit_{t_7 W} u_{s_7}^* = \unit_{t_7 W}$.
Continuing like this using at every stage the fact that $K^{21}E \subseteq F\setminus D$, 
we get for $i$ running from $7$ down to $1$ a set $t_i \cdots t_7 W\subseteq _{K^{3(8-i)}E} Z^F$
such that $u_{t_i}\unit_{t_{i+1} \cdots t_7 W} u_{t_i}^* = \unit_{t_i \cdots t_7 W}$
and $u_{s_i}\unit_{t_i \cdots t_7 W} u_{s_i}^* = \unit_{t_i \cdots t_7 W}$.
Applying all of these equalities in succession and writing $t=t_1 \cdots t_7$, we obtain
\begin{align*}
u_s \unit_W u_s^* = u_t \unit_W u_t^* = \unit_{tW} ,
\end{align*}
with the conjugation being implemented by a permutation of the levels of the tower $(F,SC)$ (as identified with $F$) 
in $\sK^{21}$, which by \eqref{I-A3} means that $tW \subseteq W\sqcup R$. It follows that
\begin{align*}
u_s \unit_W = (u_s \unit_W u_s^* )u_s = \unit_{W\sqcup R} (u_s \unit_W u_s^* )u_s = \unit_{W\sqcup R} u_s \unit_W
\end{align*}
and hence
\begin{align*}
a_1\unit_W 
= \sum_{s\in L_1} \rho_s u_s \unit_W 
= \unit_{W\sqcup R} \sum_{s\in L_1} \rho_s u_s \unit_W
= \unit_{W\sqcup R} a_1 \unit_W ,
\end{align*}
which verifies \eqref{I-factone}. 

The verification of \eqref{I-facttwo} is similar and proceeds as follows. Let $s = s_1 t_1 s_2 t_2 \cdots s_7 t_7 s_8 \in L_1$
be as before. Since $V\subseteq_E Z^F$ we have $R = Z^F \setminus V \subseteq_E F$, and so, using the fact 
that $K^{21} E \subseteq F\setminus D$ as above, we successively produce for $i$ running from $7$ down to $1$ a set
$t_i \cdots t_7 R \subseteq_{K^{3(8-i)} E} Z^F$ such that $u_{t_i}\unit_{t_{i+1} \cdots t_7 R} u_{t_i}^* = \unit_{t_i \cdots t_7 R}$
and $u_{s_i}\unit_{t_i \cdots t_7 R} u_{s_i}^* = \unit_{t_i \cdots t_7 R}$, yielding as before, setting $t=t_1 \cdots t_7$,
\begin{align*}
u_s \unit_R u_s^* = u_t \unit_R u_t^* = \unit_{tR} ,
\end{align*}
with the conjugation being implemented by 
a permutation of the levels of the tower $(F,SC)$ (as identified with $F$) in $\sK^{21}$. 
By the definition of $B$ in \eqref{E-RB} we see that $tR \subseteq R\sqcup B$, in which case
\begin{align*}
u_s \unit_R = (u_s \unit_R u_s^* )u_s = \unit_{R\sqcup B} (u_s \unit_R u_s^* )u_s = \unit_{R\sqcup B} u_s \unit_R
\end{align*}
and hence
\begin{align*}
a_1\unit_R 
= \sum_{s\in L_1} \rho_s u_s \unit_R 
= \unit_{R\sqcup B} \sum_{s\in L_1} \rho_s u_s \unit_R
= \unit_{R\sqcup B} a_1 \unit_R ,
\end{align*}
verifying \eqref{I-facttwo}.

Now let $2\leq i,i' \leq n$ with $i\neq i'$ and $2\leq j,j' \leq n$. Write $\alpha = (i,j)$ and $\beta = (i',j' )$ for brevity.
To complete the picture of the matrix we will check that $\unit_{V_\alpha} uwauv \unit_{V_\beta} = 0$.
To do this we will show that this equality holds when $V_\beta$ is replaced by either $V_\beta \setminus Y_2$ or $V_\beta \cap Y_2$.
We make use of \eqref{I-factone} and \eqref{I-facttwo} for the set $W := \bigsqcup_{k=1}^n V_{i',k}$.
Observe that
\begin{align*}
\unit_{V_\alpha} uw a_1\unit_W
&\overset{\mathclap{\eqref{I-factone}}}{=} \unit_{V_\alpha} uw \unit_{R\sqcup W} a_1\unit_W \\
&= \unit_{V_\alpha} uw \unit_{(R\setminus Y_3 )\sqcup W} a_1\unit_W \\
&= \unit_{V_\alpha} u\unit_{R\sqcup W} w \unit_{(R\setminus Y_3 )\sqcup W} a_1\unit_W \\
&= \unit_{V_\alpha} \unit_{R\sqcup W} u \unit_{R\sqcup W} w \unit_{(R\setminus Y_3 )\sqcup W} a_1\unit_W \\
&= 0
\end{align*}
and
\begin{align*}
1_{V_\alpha}uwa_1 1_R&\overset{\mathclap{\eqref{I-facttwo}}}{=}1_{V_\alpha}uw1_{R\sqcup B}a_1 1_R\\
&=1_{V_\alpha}uw1_{(R\setminus Y_3)\sqcup B}a_1 1_R\\
&=1_{V_\alpha}u1_{R}w1_{(R\setminus Y_3)\sqcup B}a_1 1_R\\
&=1_{V_\alpha}1_{R}w1_{(R\setminus Y_3)\sqcup B}a_1 1_R\\
&=0.
\end{align*}
In particular,
\[
1_{V_\alpha}uwa_11_{W\sqcup R}=0
\]
and thus
\begin{align*}
1_{V_\alpha}uwa_1 uv1_{V_\beta}&=1_{V_\alpha}uwa_1 u1_{V_\beta\sqcup R}v1_{V_\beta}\\
&=1_{V_\alpha}uwa_1 1_{W\sqcup R}u1_{V_\beta\sqcup R}v1_{V_\beta}\\
&=0.
\end{align*}
Therefore the matrix for $b$ has the desired form.

The construction of the commutative C$^*$-subalgebra $N\subseteq \Cb \fA_S$
that gave rise to our Bernoulli base $\{ -1,1 \}^d$ was such that any two of its projections
with the same value on the trace $\tau$ are Murray--von Neumann equivalent by some partial isometry
in $\Cb \fA_S$. Since for each $\kappa\in \{ -1,1\}^d$ the set $\tilde{A}_\kappa$ was defined as
$\pi_E^{-1} \pi_E (A_\kappa )\subseteq Z^F$ and membership in $A_\kappa$ is determined 
by the coordinates over $E$, it follows that the indicator functions of the sets
$\tilde{A}_\kappa$ (i.e., the sets $V_{i,j}$) are pairwise Murray--von Neumann equivalent in 
$\Cb \fA_S^F \cong (\Cb \fA_S )^{\otimes F}$.
Thus for every $\alpha = (i,j)$ and $\beta = (i',j' )$
there is a partial isometry $e_{\alpha , \beta}$ in $\Cb \fA_S^F$ with source projection $1_{V_\beta}$ and
range projection $1_{V_\alpha}$.
Set $\Theta = \{2,\dots , n\} \times \{ 1,\dots , n \}$ and let $\varphi : \Theta \to \{ 1,\dots , n(n-1) \}$
be the bijection that gives the lexicographic ordering to $\Theta$. Pick a permutation $\rho_1'$ 
of $\{1,\dots,n(n-1) \}$ that for every $i=1,\dots , n-1$ shifts the interval
$\{(i-1)(n-1)+1,\dots,i(n-1)\}$ by $i$ to the right, and a permutation $\rho_2'$
of $\{1,\dots , n(n-1) \}$ that for every $i=1,\dots , n-1$ shifts the interval
$\{(i-1)n+2,\dots in\}$ by $n-i-1$ to the right. Set $\rho_1 = \varphi^{-1}\circ\rho_1' \circ\varphi$
and $\rho_2 = \varphi^{-1}\circ\rho_2' \circ\varphi$, which are permutations of $\Theta$.

We now define in $\Cb \fA_S^F$ the two unitaries
\begin{align*}
z_1 = \unit_R + \sum_{\alpha\in\Theta} e_{\alpha , \rho_1 (\alpha )}, \hspace*{5mm}
z_2 = \unit_R + \sum_{\alpha\in\Theta} e_{\alpha , \rho_2 (\alpha )}
\end{align*}
In the illustrative case $n=4$, multiplying $b$ by $z_2$ on the right we obtain a matrix of the form
\begin{gather*}
\mbox{\footnotesize 
$\left[
\begin{array}{c | c c c c | c c c c | c c c c}
0 & 0 & 0 & 0 & \entry & \entry & \entry & \entry & \entry & \entry & \entry & \entry & \entry \\
\hline
0 & 0 & 0 & 0 & 0 & 0 & 0 & 0 & 0 & 0 & 0 & 0 & 0\\
0 & 0 & 0 & 0 & \entry & \entry & \entry & 0 & 0 & 0 & 0 & 0 & 0\\
0 & 0 & 0 & 0 & \entry & \entry & \entry & 0 & 0 & 0 & 0 & 0 & 0\\
0 & 0 & 0 & 0 & \entry & \entry & \entry & 0 & 0 & 0 & 0 & 0 & 0\\
\hline
0 & 0 & 0 & 0 & 0 & 0 & 0 & 0 & 0 & 0 & 0 & 0 & 0\\
0 & 0 & 0 & 0 & 0 & 0 & 0 & \entry & \entry & \entry & 0 & 0 & 0\\
0 & 0 & 0 & 0 & 0 & 0 & 0 & \entry & \entry & \entry & 0 & 0 & 0\\
0 & 0 & 0 & 0 & 0 & 0 & 0 & \entry & \entry & \entry & 0 & 0 & 0\\
\hline
0 & 0 & 0 & 0 & 0 & 0 & 0 & 0 & 0 & 0 & 0 & 0 & 0\\
0 & 0 & 0 & 0 & 0 & 0 & 0 & 0 & 0 & 0 & \entry & \entry & \entry\\
0 & 0 & 0 & 0 & 0 & 0 & 0 & 0 & 0 & 0 & \entry & \entry & \entry\\
0 & 0 & 0 & 0 & 0 & 0 & 0 & 0 & 0 & 0 & \entry & \entry & \entry
\end{array}
\right]$} 
\end{gather*}
and then multiplying $bz_2$ on the left by $z_1$ produces a matrix of the form
\begin{gather*}
\mbox{\footnotesize 
$\left[
\begin{array}{c | c c c c | c c c c | c c c c}
0 & 0 & 0 & 0 & \entry & \entry & \entry & \entry & \entry & \entry & \entry & \entry & \entry \\
\hline
0 & 0 & 0 & 0 & \entry & \entry & \entry & 0 & 0 & 0 & 0 & 0 & 0\\
0 & 0 & 0 & 0 & \entry & \entry & \entry & 0 & 0 & 0 & 0 & 0 & 0\\
0 & 0 & 0 & 0 & \entry & \entry & \entry & 0 & 0 & 0 & 0 & 0 & 0\\
0 & 0 & 0 & 0 & 0 & 0 & 0 & \entry & \entry & \entry & 0 & 0 & 0\\
\hline
0 & 0 & 0 & 0 & 0 & 0 & 0 & \entry & \entry & \entry & 0 & 0 & 0\\
0 & 0 & 0 & 0 & 0 & 0 & 0 & \entry & \entry & \entry & 0 & 0 & 0\\
0 & 0 & 0 & 0 & 0 & 0 & 0 & 0 & 0 & 0 & \entry & \entry & \entry\\
0 & 0 & 0 & 0 & 0 & 0 & 0 & 0 & 0 & 0 & \entry & \entry & \entry\\
\hline
0 & 0 & 0 & 0 & 0 & 0 & 0 & 0 & 0 & 0 & \entry & \entry & \entry\\
0 & 0 & 0 & 0 & 0 & 0 & 0 & 0 & 0 & 0 & 0 & 0 & 0\\
0 & 0 & 0 & 0 & 0 & 0 & 0 & 0 & 0 & 0 & 0 & 0 & 0\\
0 & 0 & 0 & 0 & 0 & 0 & 0 & 0 & 0 & 0 & 0 & 0 & 0
\end{array}
\right]$} .
\end{gather*}
The matrix representation of $z_1 b z_2$ is then strictly upper triangular and hence 
$(z_1 b z_2 )^{1 + n(n-1)} = 0$, so that we obtain the desired nilpotence.

\end{document}